\documentclass[12pt]{amsart}
\numberwithin{equation}{section}
\usepackage{amssymb}
\usepackage{amsmath}

\newtheorem{lem}{Lemma}[section]
\newtheorem{thm}{Theorem}[section]
\newtheorem{que}{Question}[section]
\newtheorem{cor}{Corollary}[section]
\theoremstyle{definition}
\newtheorem{defin}{Definition}[section]
\theoremstyle{remark}

\title[Transfer operator and Conformal Maps]{Transfer operator and Conformal Measures for a Class of Maps
Having Covering Property} 
\author{Zheng Jian-Hua}
\thanks{Supported by the grant (No. 11171170) of NSF of China and the research fund (No. 20100002110012)
for the Doctoral Program of Higher Education, Ministry of Education
of China.}
\address{Department of Mathematical Sciences, Tsinghua University, P. R. China}
\email{jzheng@math.tsinghua.edu.cn}

\begin{document}


\begin{abstract}
Let $(X,d)$ be a metric space and $X_0$ be an open and dense subset
of $X$. We develop the Walters' theory and discuss the existence of
conformal measures in terms of the Perron-Frobenius-Ruelle operator
for a continuous map $T:X_0\rightarrow X$ and the Bowen formula
about Hausdorff dimension and Poincar\'e exponent of some invariant
subsests for $T$ with some expanding property.
\end{abstract}

\maketitle

\section{Introduction and Notations}
This paper consists of two aspects. Firstly, we make a careful study
of the existences of conformal measures, invariant measures and
equilibrium states from Walters' viewpoint in \cite{Walters} for
some countable-to-one maps, which is able to be used to the complex
transcendental dynamics. Secondly, we discuss the Bowen formula
about the Poincare exponent and the Hausdorff dimension of some
invariant subsets.

We introduce basic notations which will be often used. Let
$(\widehat{X},d)$ be a compact metric space and $X$ be an open and
dense subset of $X$. For an open and dense subset $X_0$ of $X$,
consider a continuous map $T:X_0\rightarrow X$.
$\mathcal{C}(\Omega)$ will denote the set of all real-valued
continuous functions on $\Omega=\widehat{X}, X$ or $X_0$. Then
$\mathcal{C}(\widehat{X})$ is a Banach space with the supremum norm:
for $f\in \mathcal{C}(\widehat{X})$,
$\|f\|=\max\{|f(x)|:x\in\widehat{X}\}$ and
$\mathcal{C}(\widehat{X})^*$ is the dual space of
$\mathcal{C}(\widehat{X})$. For $f\in \mathcal{C}(\widehat{X})$,
$\|f\|$ is the norm of $f$ and the notation "$\rightrightarrows$"
will denote convergence under the norm. By $\mathcal{M}(\Omega)$ we
mean the set of all probability measures on the $\sigma$-algebra of
Borel sets of $\Omega=\widehat{X}$ or $X$

A $\mu\in\mathcal{M}(X)$ is called $g$-conformal measure for a
$\mu$-measurable function $g:X_0\rightarrow \mathbb{R}$ over $X_0$
if $g$ is Jacobian of $T$ with respect to $\mu$, namely, for any
Borel subset $A$ of $X_0$ such that $T$ is injective on $A$, we have
$$\mu(T(A))=\int_Ag\mathrm{d}\mu.$$
A general scheme for constructing conformal measure can be found in
Denker and Urbanski \cite{DenkerUrbanski3}, but in this paper, we
use the transfer operator to get the desired conformal measure.
Actually it is the eigenmeasure of the dual operator of the transfer
operator. The method has been used in many references, e.g., Ruelle
\cite{Ruelle}, Walters \cite{Walters}.

To guarantee the existence of the transfer operator of
$T:X_0\rightarrow X$ from $\mathcal{C}(\widehat{X})$ into itself, we
impose some conditions on $T$ and $\varphi\in \mathcal{C}(X_0)$
which are listed as follows:

\

(1a) \textsl{The set $T^{-1}(x)$ for each $x\in X$ is at most
countable.}

\

(1b) \textsl{$T$ has the uniformly covering property: there exists a
$\delta>0$ such that for each $x\in X$, $T^{-1}(B_X(x,\delta))$ can
be written uniquely as a disjoint union of a finite or countable
number of open subsets $A_i(x)\ (1\leq i\leq N\leq \infty)$ of $X_0$
and for each $i$, $T$ is a homeomorphism of $A_i(x)$ onto
$B_X(x,\delta)$, where $B_X(x,\delta)=B(x,\delta)\cap X$; For the
simplicity, we will call $A_i(x)$ injective component of $T^{-1}$
over $B_X(x,\delta)$.}

\

(1c)   \textsl{the inverse of $T$ is locally uniformly continuous:
$\forall \varepsilon>0$, $\exists \delta_0$ with $0<\delta_0<\delta$
such that for each $x\in X$ and each $y\in X_0$ with $T(y)=x$, once
$d(x,x')<\delta_0$ for $x'\in X$, we have $d(T^{-1}_y(x),
T^{-1}_y(x'))<\varepsilon,$ where $T^{-1}_y$ is the branch of the
inverse of $T$ which sends $x$ to $y$, that is to say, every
injective component of $T^{-1}$ over $B_X(x,\delta_0)$ has diameter
less than $\varepsilon$.}

\

(1d) \textsl{Let $\varphi\in \mathcal{C}(X_0)$. $\forall
\varepsilon>0$, there exists a $0<\delta_1<\delta$ such that for any
pair $x, x'\in X$, once $d(x,x')<\delta_1$, we have
$$\sum_{T(y)=x}\left|\exp(\varphi(T^{-1}_y(x)))-\exp(\varphi(T^{-1}_y(x')))\right|<\varepsilon,$$
that is,
$\sum_{T(y)=x}\left|\exp(\varphi(T^{-1}_y(x)))-\exp(\varphi(T^{-1}_y(x')))\right|\rightarrow
0$ uniformly as $d(x,x')\rightarrow 0$.}

\

An ordered pair $(T, \varphi)$ is called admissible if $T$ satisfies
(1a), (1b), (1c), (1d) and $\varphi\in \mathcal{C}(X_0)$ is summable
on $X$, that is to say,
$$\sup\left\{\sum_{T(y)=x}\exp(\varphi(y)):\ \ x\ \in
X\right\}<+\infty.$$ Then for a summable function $\varphi$ on $X$,
$$\mathcal{L}_\varphi(f)(x):=\sum_{T(y)=x}f(y)\exp(\varphi(y)), \forall x\in
X$$ is a bounded real-valued function on $X$ for a bounded
real-valued function $f$ on $X_0$. Sometimes, we write
$\mathcal{L}_{\varphi,T}$ for $\mathcal{L}_\varphi$ to emphasize
$T$. It is obvious that $T^n$ is a continuous mapping of
$T^{-n+1}X_0$ to $X$. Set
$$S_n\varphi(y)=\sum_{i=0}^{n-1}\varphi(T^i(y)),\ y\in T^{-n+1}X_0$$
and noting that $T^{-n+1}X_0\subseteq X_0$, we easily deduce
\begin{equation}\label{123}
\mathcal{L}_{\varphi,T}^n(f)(x)=\mathcal{L}_{S_n\varphi,T^n}(f)(x)=\sum_{T^n(y)=x}f(y)\exp(S_n\varphi(y)),\
x\in X,\end{equation} here and throughout the paper we denote by
$\mathcal{L}_{\varphi,T}^n$ the $n$th iterate of
$\mathcal{L}_{\varphi}=\mathcal{L}_{\varphi,T}$. Now we introduce
the pressure function. For a point $x\in X$, define
$$\overline{P_x}(T,\varphi)=\limsup_{n\rightarrow\infty}\frac{1}{n}\log\mathcal{L}^n_\varphi(1)(x),$$
$$\underline{P_x}(T,\varphi)=\liminf_{n\rightarrow\infty}\frac{1}{n}\log\mathcal{L}^n_\varphi(1)(x)$$
and if $\overline{P_x}(T,\varphi)=\underline{P_x}(T,\varphi)$, we
write the value as $P_x(T,\varphi)$, and if $P_x(T,\varphi)$ is
independent of the choice of $x$, we write the value as
$P(T,\varphi)$, which is called the pressure (function) of $T$ with
respect to $\varphi$. To guarantee the existence of the pressure
function $P(T,\varphi)$, we need the following condition:

(1e) \textsl{For arbitrary $\varepsilon>0$, there exists a $m\geq 1$
such that $T^{-m}(x)$ is $\varepsilon-$dense in $X$ for each $x\in
X$.}

A continuous map $T:X_0\rightarrow X$ satisfying (1e) is called
(topologically) backward dense. If for a fixed $\varepsilon>0$, (1e)
holds, then we call $T$ backward $\varepsilon-$dense.

The following is the first main result we shall establish.

\begin{thm}\label{thm1.1+}\ \  Let $(T,\varphi)$ be admissible and
for a sequence of positive numbers $\{K_n\}$ with
$\frac{K_n}{n}\rightarrow 0$ as $n\rightarrow\infty$, we have
\begin{equation}\label{1add}|S_n\varphi(y)-S_n\varphi(y')|\leq K_n\end{equation} whenever $y$ and $y'$
are in a component of $T^{-n}(B_X(x,\delta)), \forall x\in X$ and
$T$ is backward $\delta$-dense. Then the pressure function
$P(T,\varphi)$ of $T$ with respect to $\varphi$ exists and there
exists a $\mu\in \mathcal{M}(\widehat{X})$ such that
$\exp(-\varphi+P(T,\varphi))$ is the Jacobian of $T$ with respect to
$\mu$. From the backward dense property, $\mu$ is positive on
nonempty open sets. Finally, we have
\begin{equation}\label{add10}Ce^{-K_n}\leq{\mu(T^{-n}_x(B(T^n(x),\delta)))\over
\exp(S_n\varphi(x)-nP(T,\varphi))}\leq e^{K_n},\ \ \forall\ x\in
T^{-n}(X), \forall\ n\in\mathbb{N},
\end{equation}
for a constant $C>0$ only depending on $\delta$, where
$T^{-n}_x(B(T^n(x),\delta))$ is the component of
$T^{-n}(B(T^n(x),\delta))$ containing $x$ and on it $T^n$ is
injective.
\end{thm}

A component $U$ of $T^{-n}(B_X(x,\delta))$ means that $T^n$ maps $U$
onto $B_X(x,\delta)$ and $U$ cannot be written into $U=U_1\cup U_2$
such that $U_i (i=1,2)$ are open and disjoint and $T^n$ maps $U_i$
onto $B_X(x,\delta)$. We remark on (\ref{1add}). If $K_n$ is chosen
to be a fixed constant $C$ and $y, y'$ have the distance of Bowen
metric $d_n$ less than $\delta$ and $X_0=X=\widehat{X}$, then the
condition (\ref{1add}) is known as Bowen condition (\cite{Bowen},
\cite{HaydnRuelle}, \cite{Walters2}). Here
$d_n(y,y')=\max\{d(f^j(y),f^j(y')):0\leq j\leq n-1\}.$ About
(\ref{add10}), generally, under the assumption of Theorem
\ref{thm1.1+} we cannot assert that $T^n$ is injective on
$T^{-n}_x(B(T^n(x),\delta))$. If $T^n$ is injective on
$T^{-n}_x(B(T^n(x),\delta))$ for all $x\in X$ and $n\in\mathbb{N}$,
then $\mu$ is a Gibbs state as in the definition given in
\cite{PrzytyckiUrbanski}. A transcendental parabolic meromorphic
function on the Riemann sphere satisfies the assumptions of Theorem
\ref{thm1.1+} with $K_n=O(\log n)$ (cf. \cite{Zheng}).

Next we consider the existence of invariant measure equivalent to
the conformal measure $\mu$. To the end, we need an expanding
condition:

(1c*) \textsl{whenever $y$ and $y'$ are in one of $A_i(x)$'s, we
have
$$d(T(y),T(y'))\geq d(y,y').$$}

Let $\mathcal{M}(\Omega,T)$ be the set of all invariant measures in
$\mathcal{M}(\Omega)$ for $T$. For $\Omega=X$, that $\mu\in
\mathcal{M}(\Omega,T)$ means $\mu(T^{-1}B)=\mu(B)$ for any
measurable subset $B$ of $X$, that is, $T$ preserves the measures
$\mu|_{X_0}$ and $\mu|_X$ and in this case, $\mu|_X(X\setminus
X_0)=0$. For $\mu\in\mathcal{M}(\Omega)$ and
$f\in\mathcal{C}(\Omega)$, set $\mu(f)=\int fd\mu$ and for
$h\in\mathcal{C}(\Omega)$, define $h\cdot\mu$ by
$(h\cdot\mu)(f)=\mu(hf), f\in\mathcal{C}(\Omega)$. Let $\mathcal{B}$
be $\sigma$-algebra of Borel sets of $X$. If $\mathcal{D}$ is a
subalgebra of $\mathcal{B}$ and $\mu\in\mathcal{M}(X)$, then
$E_\mu(f/\mathcal{D})$ (respectively,
$I_\mu(\mathcal{B}/\mathcal{D})$) is the conditional expectation
(respectively, information) of $f$ (respectively, $\mathcal{B}$)
with respect to $\mathcal{D}$. The second result we shall establish
is a modifying version of the main results in Walters
\cite{Walters}.

\begin{thm}\label{thm1.1}  Let the pair
$(T,\varphi)$ be admissible and for some fixed $N\in\mathbb{N}$,
$T^N$ satisfy (1c*) and (1g) for some $\delta_N$ and (1e). Then

(1) there exist $\mu\in\mathcal{M}(X)$ and $\lambda>0$ such that
$\mathcal{L}_\varphi^*(\mu)=\lambda\mu$ and $\lambda e^{-\varphi}$
is the Jacobian of $T$ with respect to $\mu$. The pair
$(\lambda,\mu)$ is uniquely determined by the conditions $\lambda>0,
\mu\in\mathcal{M}(X)$ and $\mathcal{L}_\varphi^*(\mu)=\lambda\mu$;

(2) there exists a $h\in \mathcal{C}(\widehat{X})$ with $h>0$ such
that $\mu(h)=1,\ \mathcal{L}_\varphi(h)=\lambda h$;

(3) $h$ satisfies $h(x)\leq e^{C_\varphi(x,x')}h(x')$ and $h$ is
uniquely determined by this condition and the properties $h>0,
\mu(h)=1$ and $\mathcal{L}_\varphi(h)=\lambda h$;

(4) $\lambda^{-n}\mathcal{L}_\varphi^n(f)\rightrightarrows h\cdot
\mu(f), \forall f\in \mathcal{C}(\widehat{X})$;

(5) $m=h\cdot\mu$ is a Gibbs invariant measure for $T$ and
$\mathcal{L}^*_\psi(m)=m$, where $$\psi=\varphi-\log\lambda+\log
h-\log h\circ T.$$

(6)
$\log\lambda=P(T,\varphi)=\sup\{\nu(I_\nu(\mathcal{B}|T^{-1}\mathcal{B})+\varphi):\nu\in\mathcal{M}(X,T)\}$
and $m$ is the equilibrium state.

(7) $m$ and $\mu$ are positive on nonempty open sets and have no
atoms.
\end{thm}

This modifying version of the Walters results in \cite{Walters}
makes us be able to establish the results on thermodynamic formalism
of some transcendental meromorphic functions on $\mathbb{C}$ over
their Julia sets. Actually, a meromorphic function itself may not be
expanding over its Julia set, but the $N$th iterate of it may have
the strict expanding property, that is, satisfies (1c*) for some
$N$. In terms of Theorem \ref{thm1.1}, it is sufficient to know that
$(f,-s\log f^\times)$ is admissible over its Julia set where $s$ is
the Poincare exponent.




\section{Conformal Measures and Ruelle-type Theorem}
In this section, we develop main results in Walters \cite{Walters}
for our purpose. Let $T:X_0\rightarrow X$ be continuous and satisfy
(1a), (1b) and (1c). We first establish the transfer operator or the
Perron-Frobenius-Ruelle operator of $\mathcal{C}(\widehat{X})$ to
itself. For the case when $X_0=X=\widehat{X}$, this is trivial. Next
through the eigenvalue and eigenmeasure of the dual operator of the
transfer operator, we seek the desired conformal (and invariant as
well) measures and discuss the thermodynamic properties of the
measures. These types of results are known as Ruelle-type Theorem
(See \cite{Ruelle1}, \cite{Keane}, \cite{Walters3}, \cite{Fan} and
\cite{FanJiang}).

We make a remark on the conditions (1a), (1b) and (1c). (1a) follows
from (1b), but (1c) does not follows from (1b) if $N=\infty$. Every
branch of $T^{-1}$ is continuous on $B_X(x,\delta)$, but $T^{-1}$ on
$B_X(x,\delta)$ may have a countable number of branches and thus we
cannot generally confirm that (1c) holds under (1b). Below we give
$T$ a simple condition such that (1c) holds under (1b).

\begin{lem}\label{lem2.1}\ \ Let $T$ satisfy (1b) with $X=\widehat{X}$.
Assume that (*) for arbitrary $\varepsilon>0$, we have a $0<\eta\leq
\varepsilon$ such that for each $x\in X\setminus X_0$, $\partial
B(x,\eta)\subset X_0$. Then the inverse of $T$ is locally uniformly
continuous, that is, $T$ satisfies (1c).\end{lem}

{\bf Proof.} We are arbitrarily given a $\varepsilon>0$. For a point
$x\in X$, let $T^{-1}_j (1\leq j\leq N\leq \infty)$ be the branch of
$T^{-1}$ of $B_X(x,\delta)$ onto $A_j(x)$.

We claim that for each $x_0\in X\setminus X_0$, $\partial B(x_0,
\eta)$ for $0<\eta<\varepsilon/2$ and $T^{-1}_j(B_X(x,\delta))$
intersect only for finitely many $j$. Suppose that fails and then
for a sequence $\{n_k\}$, $\partial B(x_0, \eta)\cap
T^{-1}_{n_k}(B_X(x,\delta))\not=\emptyset.$ From each of these
intersecting sets, take a point $z_{n_k}$ and so $T(z_{n_k})\in
B_X(x,\delta)$ and $z_{n_k}\in\partial B(x_0,\eta)\subset X_0$.
Since $\partial B(x_0,\eta)$ and $\overline{B_X(x,\delta)}$ are
compact, we can assume that $z_{n_k}\rightarrow z_0\in \partial
B(x_0,\eta)$ and $T(z_{n_k})\rightarrow w\in B_X(x,\delta)$ as
$k\rightarrow \infty$ (otherwise let us shrink $\delta$ a little
bit). Noting that $T$ is continuous at $z_0$, we have
$T(z_{n_k})\rightarrow T(z_0)=w$ as $k\rightarrow \infty$ and so
$z_0\in T^{-1}_{j_0}(B_X(x,\delta))$ for some $j_0$ and furthermore
for a $c>0$, $B(z_0,c)\cap X_0\subset T^{-1}_{j_0}(B_X(x,\delta))$.
This contradicts that $z_{n_k}\rightarrow z_0$ as $k\rightarrow
\infty$, because $T^{-1}_j(B_X(x,\delta))$ does not intersect each
other and so we have proved the claim.

We can take finitely many points $x_i\in X\setminus X_0\ (1\leq
i\leq M(\varepsilon))$ such that
$$X\setminus X_0\subseteq\cup_{i=1}^MB_X(x_i,\eta).$$
For $\forall\ x\in X$, all but at most finitely many $T^{-1}_j(x)$
lie in $\cup_{i=1}^MB_X(x_i,\eta).$ In terms of the claim, with the
possible exception of finitely many $j$, $T^{-1}_j(B_X(x,\delta))$
lies in one of $B_X(x_s,\eta) (1\leq s\leq M)$ and so $${\rm
diam}(T^{-1}_j(B_X(x,\delta)))<2\eta<\varepsilon.$$ Thus we can
choose a $0<\delta_x<\delta$ such that ${\rm
diam}(T^{-1}_j(B_X(x,\delta_x))))<\varepsilon$ for all $j$. Since
$X=\widehat{X}$ is compact, we have proved (1c).\qed

For the case of that $X_0=X=\widehat{X}$, a continuous surjection
$T:\widehat{X}\rightarrow\widehat{X}$ is a local homeomorphism, that
is to say, for $x\in\widehat{X}$ there exists an open neighborhood
$V(x)$ of $x$ such that $T(V(x))$ is open and $T: V(x)\rightarrow
T(V(x))$ is a homeomorphism, if and only if (1b) holds; And (1b)
implies (1c). These were proved by Eilenberg (See Page 31 of
\cite{AokiHiraide}).

\begin{thm}\label{thm2.1}\ \  Let $(T,\varphi)$ be admissible. Then
$\mathcal{L}_\varphi$ can be extended to a linear operator of
$C(\widehat{X})$ to itself, which is still denoted by
$\mathcal{L}_\varphi$, there exists a $\mu\in
\mathcal{M}(\widehat{X})$ such that
$\mathcal{L}_\varphi^*(\mu)=\lambda\mu,\
\lambda=\mathcal{L}_\varphi^*(\mu)(1)>0$, where
$\mathcal{L}_\varphi^*$ is the dual operator of
$\mathcal{L}_\varphi$, and the following statements hold:

(1) $\lambda\exp(-\varphi)$ is the Jacobian of $T$ with respect to
$\mu$;

(2) $\mu$ is positively nonsingular and nonsingular for $T$, that
is, $\mu\circ T\ll \mu$ and $\mu\circ T^{-1}\ll\mu$.
\end{thm}

Sometimes, we write $\mu_\varphi$ for $\mu$ and $\lambda_\varphi$
for $\lambda$ in Theorem \ref{thm2.1}. We remark that the results
(1) and (2) in Theorem \ref{thm2.1} follow from the formula
$\mathcal{L}_\varphi^*(\mu)=\lambda\mu$. Although the first result
in Theorem \ref{thm2.1} is new, Theorem \ref{thm2.1} is essentially
due to Walters \cite{Walters}, while Walters obtained the result
with the condition (1c) replaced by that $T$ does not decrease any
distance on every injective component of $T^{-1}$ over
$B_X(x,\delta)$ for each $x\in X$, that is, (1c*). It is obvious
that the condition (1c) can be derived from the Walters' condition
(1c*).

We first consider the existence of $P(T,\varphi)$. We recall that a
continuous map $T:X_0\rightarrow X$ satisfying (1e) is called
(topologically) backward dense. "Topological backward dense" has
something to do with topological transitive, exact and mixing. In
fact, "Topological backward dense" is equivalent to that $(\dag)$
for any open set $U$ of $X$, there exists a $N$ such that $T^N(U\cap
T^{-N}(X))=X$. Let us prove that. Assume "Topological backward
dense". Take a ball $B(a,\varepsilon)\subset U$. There exists a $N$
such that $\forall\ x\in X$, $T^{-N}(x)\cap
B(a,\varepsilon)\not=\emptyset.$ This implies that
$X=T^N(B(a,\varepsilon)\cap T^{-N}(X))\subseteq T^N(U\cap
T^{-N}(X))\subseteq X,$ and so $T^N(U\cap T^{-N}(X))= X.$
Conversely, assume $(\dag)$. For any $\varepsilon>0$, since
$\widehat{X}$ is compact, we have
$X=\cup_{j=1}^qB_X(x_j,\varepsilon/2)$ and therefore, there exists a
$N$ such that $T^N(B_X(x_j,\varepsilon/2)\cap T^{-N}(X))=X\ (1\leq
j\leq q)$. This yields that $\forall x\in X$, we can take a point
$y_j\in T^{-N}(x)\cap B(x_j,\varepsilon/2)\not=\emptyset$ for each
$j$. Certainly, $B(x_j,\varepsilon/2)\subset B(y_j,\varepsilon)$ and
so $X=\cup_{j=1}^qB_X(y_j,\varepsilon)$, that is, $T^{-N}(x)$ is
$\varepsilon$-dense in $X$.

\begin{thm}\label{6add}\ \ Let $T:X_0\rightarrow X$ satisfy (1a) and
$\varphi\in\mathcal{C}(X_0)$. Assume that (\ref{1add}) holds for a
$\delta>0$ and $T$ is a backward $\delta$-dense. If for an $a\in X$,
$\overline{P_a}(T,\varphi)<\infty$, then the pressure function
$P(T,\varphi)$ exists.
\end{thm}

{\bf Proof.}\ \ Since $\widehat{X}$ is compact, there exist finitely
many points $x_i\in X (i=1,2,...,q)$ such that
$\widehat{X}=\cup_{i=1}^qB(x_i,\delta)$. In view of (1e) for
$\delta$, for some $p$, $f^{-p}(a)$ is $\delta$-dense in $X$. Since
$\overline{P_a}(T,\varphi)<\infty$, for all $n$,
$\mathcal{L}^n_\varphi(a)$ is finite. For each $i$ take a point
$\widehat{x}_i\in T^{-p}(a)\cap B_X(x_i,\delta)$. Set
$A_p=\max\{-S_p\varphi(\widehat{x}_i):1\leq i\leq q\}$.

Take $\forall\ m\in\mathbb{N}$ and $\forall\ x\in X$. Then $x\in
B_X(x_i,\delta)$ for some fixed $i$. $\forall\ y\in T^{-m}(x)$,
$\exists\ y'\in T^{-m}(\widehat{x}_i)$ such that $y, y'$ are in a
component of $T^{-m}(B_X(x_i,\delta))$. Then in view of
(\ref{1add}), we have $|S_m\varphi(y)-S_m\varphi(y')|\leq K_m$ and
\begin{eqnarray}\label{7add}
\mathcal{L}_\varphi^{m+p}(1)(a)&=&\sum_{T^p(w)=a}\sum_{T^m(y)=w}e^{S_p\varphi(w)}e^{S_m\varphi(y)}\nonumber\\
&\geq&e^{S_p\varphi(\widehat{x}_i)}\sum_{T^m(y')=\widehat{x}_i}e^{S_m\varphi(y')}\nonumber\\
&\geq&e^{-A_p-K_m}\mathcal{L}_\varphi^m(1)(x).\end{eqnarray} In
particular, we have $\mathcal{L}_\varphi^p(1)(x)\leq
e^{A_p+K_p}\mathcal{L}_\varphi^{2p}(1)(a),\ \forall\ x\in X.$ Thus
we have
$$\mathcal{L}_\varphi^{m+p}(1)(a)=\sum_{T^m(w)=a}e^{S_m\varphi(w)}\mathcal{L}_\varphi^p(1)(w)\leq
e^{A_p+K_p}\mathcal{L}_\varphi^{2p}(1)(a)\mathcal{L}_\varphi^m(1)(a).$$

For $\forall\ n, m\in \mathbb{N}$ with $n\geq m$, we have
\begin{eqnarray*}
\mathcal{L}^{n+m}_\varphi(1)(a)&=&\sum_{T^n(w)=a}e^{S_n\varphi(w)}\mathcal{L}^m_\varphi(1)(w)\\
&\leq&
e^{A_p+K_m}\mathcal{L}_\varphi^{m+p}(1)(a)\mathcal{L}_\varphi^n(1)(a)\\
&\leq&e^{2A_p+K_p+K_m}\mathcal{L}_\varphi^{2p}(1)(a)\mathcal{L}_\varphi^n(1)(a)\mathcal{L}_\varphi^m(1)(a).
\end{eqnarray*}

Set $a_n=\log\mathcal{L}^n_\varphi(1)(a)$. The above inequality
implies that for $m\leq n$, we have $$a_{n+m}\leq a_n+a_m+K_m+C,$$
where $C=2A_p+K_p+\log\mathcal{L}_\varphi^{2p}(1)(a)$. For any fixed
$m$, we can write $n=km+i$ with $0\leq i<m$. Thus
\begin{eqnarray*}\frac{a_n}{n}&\leq &\frac{a_{km}}{km}+\frac{a_i+K_i+C}{n}\\
&\leq &\frac{ka_{m}+kK_m+kC}{km}+\frac{a_i+K_i+C}{n},
\end{eqnarray*}
and
$$\limsup_{n\rightarrow\infty}\frac{a_n}{n}\leq
\frac{a_{m}}{m}+\frac{K_m}{m}+\frac{C}{m},\ \forall\
m\in\mathbb{N}$$ so that
$$\limsup_{n\rightarrow\infty}\frac{a_n}{n}\leq
\liminf_{m\rightarrow\infty}\frac{a_{m}}{m}.$$ This implies that
$P_a(T,\varphi)$ exists. In view of (\ref{7add}), $\forall\ x\in X$,
$\overline{P_x}(T,\varphi)\leq P_a(T,\varphi)<\infty$. Thus the
above argument yields that $P_x(T,\varphi)$ exists and
$P_x(T,\varphi)\geq P_a(T,\varphi)$, and therefore, $P_x(T,\varphi)=
P_a(T,\varphi),\ \forall\ x\in X.$ \qed

Now let us consider the possible relation between the eigenvalue
$\lambda$ of the transfer operator $\mathcal{L}^*_\varphi$ and the
pressure $P(T,\varphi)$. To the end, we consider the iterates of
$\mathcal{L}_\varphi$ and $\mathcal{L}_\varphi^*$.

\begin{lem}\label{lem2.2} Let $(T,\varphi)$ be admissible. Then for
each fixed positive integer $N$, $(T^N, S_N\varphi)$ is
admissible.\end{lem}

{\bf Proof.}\ It is obvious that $T^N$ satisfies (1a), (1b) and (1c)
for some $\delta_N>0$ in the place of $\delta$. Here we first of all
check that $S_N\varphi$ is summable on $X$ for $T^N$. Set
$$K=\sup\left\{\sum_{T(y)=x}\exp(\varphi(y)):\ x\in
X\right\}<+\infty.$$ Then for each $x\in X$, we have
\begin{eqnarray*}
\sum_{T^N(y)=x}\exp(S_N\varphi(y))&=&\sum_{T^{N-1}(w)=x}\exp(S_{N-1}\varphi(w))
\sum_{T(y)=w}\exp(\varphi(y))\\
&\leq& K\sum_{T^{N-1}(w)=x}\exp(S_{N-1}\varphi(w))\leq
K^N.\end{eqnarray*}

Next we check (1d) for $(T^N, S_N\varphi)$, that is,
$$\sum_{T^N(y)=x}\left|\exp(S_N\varphi(T^{-N}_y(x)))-\exp(S_N\varphi(T^{-N}_y(x')))\right|\rightarrow
0,\ {\rm as}\ d(x,x')\rightarrow 0,$$ here $T^{-N}_y$ is the branch
of $T^{-N}$ on $B(x,\delta_N)$ which sends $x$ to $y$ and $\delta_N$
is determined in (1b) for $T^N$. Let us prove it by induction. We
assume that the result holds for $N$ and consider the case $N+1$. We
introduce some notations: for a pair $y$ and $w$ with $T^j(y)=w$,
$T^{-j}_{y,w}$ is the branch of $T^{-j}$ sending $w$ to $y$. For any
pair $x, x'\in X$ with $d(x,x')<\delta_{N+1}$, we have
$$\sum_{T^{N+1}(y)=x}\left|e^{S_{N+1}\varphi(y)}-e^{S_{N+1}\varphi(y')}\right|
=\sum_{T^{N}(w)=x}\sum_{T(y)=w}\left|e^{S_{N}\varphi(w)+\varphi(y)}-e^{S_{N}\varphi(w')+\varphi(y')}\right|$$
\begin{eqnarray*}&\leq &
\sum_{T^{N}(w)=x}\sum_{T(y)=w}\left|e^{S_{N}\varphi(w)+\varphi(y)}-e^{S_{N}\varphi(w)+\varphi(y')}\right|\\
&+&\sum_{T^{N}(w)=x}\sum_{T(y)=w}\left|e^{S_{N}\varphi(w)+\varphi(y')}-e^{S_{N}\varphi(w')+\varphi(y')}\right|\\
&\leq &\sum_{T^{N}(w)=x}e^{S_{N}\varphi(w)}\sum_{T(y)=w}\left|e^{\varphi(y)}-e^{\varphi(y')}\right|\\
&+&\sum_{T^{N}(w)=x}\left|e^{S_{N}\varphi(w)}-e^{S_{N}\varphi(w')}\right|\sum_{T(y)=w}e^{\varphi(y')}\\
&\leq &
K^N\sup_{T^N(w)=x}\sum_{T(y)=w}\left|e^{\varphi(y)}-e^{\varphi(y')}\right|\\
&+&K\sum_{T^{N}(w)=x}\left|e^{S_{N}\varphi(w)}-e^{S_{N}\varphi(w')}\right|
\rightarrow 0\end{eqnarray*} as $d(x,x')\rightarrow 0$ and so
$d(w,w')\rightarrow 0$, where $w'=T^{-N}_w(x')$ with $T^N(w)=x$ and
$y'=T^{-N-1}_y(x')$ and $T^{-N-1}_y=T^{-1}_{y,w}\circ T^{-N}_w$.
Lemma \ref{lem2.2} is proved. \qed

Under the assumption of Theorem \ref{thm2.1}, in terms of Lemma
\ref{lem2.2}, the $n$th iterates of $\mathcal{L}_\varphi$ and of
$\mathcal{L}^*_\varphi$ exist and so we have
$$\lambda^n={\mathcal{L}^*}^n_\varphi(\mu)(1)=\mu(\mathcal{L}_\varphi^n(1))$$
and therefore noting that $\mathcal{L}_\varphi(1)\in
\mathcal{C}(\widehat{X})$ implies that
\begin{equation}\label{2.2+}\inf_{x\in
X}\{\mathcal{L}^n_\varphi(1)(x)\}\leq\lambda^n\leq \sup_{x\in
X}\{\mathcal{L}^n_\varphi(1)(x)\}.\end{equation}

Obviously, the following condition is enough to confirm that
$\log\lambda=P(T,\varphi)$: there exist a sequence of positive
number $\{K_n\}$ with $K_n/n\rightarrow 0$ as $n\rightarrow\infty$
such that for any pair $x, x'\in X$,
\begin{equation}\label{2.2}e^{-K_n}\mathcal{L}^n_\varphi(1)(x')\leq
\mathcal{L}^n_\varphi(1)(x)\leq
e^{K_n}\mathcal{L}^n_\varphi(1)(x').\end{equation}

Actually, the inequality (\ref{2.2}) holds if (\ref{1add}) holds
whenever $y$ and $y'$ are in a component of $T^{-n}(B_X(x,\delta)),
\forall x\in X$ with $X$ connected, that is, for any two $x, x'\in
X$, there exist finitely many $x_i,\ i=0,1,...,q$ with $x_0=x$ and
$x_m=x$ such that $B_X(x_i,\delta/2)\cap
B_X(x_{i+1},\delta/2)\not=\emptyset$. This is a way, in view of the
connected property of $X$, to go from local property in (\ref{1add})
to whole property in (\ref{2.2}). Another way to realize the step is
the backward dense.

{\bf The Proof of Theorem \ref{thm1.1+}} is completed by Theorem
\ref{thm2.1} and the following Lemma \ref{2add} and Lemma
\ref{add11}.

\begin{lem}\label{2add} \ Let $T, \varphi,\lambda$ and $\mu$ be as
in Theorem \ref{thm2.1}. Then

(1) if (\ref{1add}) holds, we have
\begin{equation}\label{3add}\log\lambda=\max\{\overline{P_x}(T,\varphi): x\in X\ {\rm such\ that}\
\mu(B(x,\delta))>0\};\end{equation}

(2) if, in addition, $T$ is backward $\delta$-dense, we have
$\log\lambda=P(T,\varphi).$ Furthermore, $\lambda$ is unique
eigenvalue of
$\mathcal{L}_\varphi^*:\mathcal{M}(\widehat{X})\rightarrow
\mathcal{M}(\widehat{X}).$
\end{lem}

{\bf Proof.} Obviously, (2) follows from (1) and Theorem \ref{6add}.
So we only prove (1) here. We write $\gamma$ for the right of
(\ref{3add}). For $x\in X$ with $\mu(B(x,\delta))>0$, we have
$$\lambda^n\geq \mu(B(x,\delta))\inf_{a\in
B_X(x,\delta)}\mathcal{L}_\varphi^n(1)(a)\geq
\mu(B(x,\delta))e^{-K_n}\mathcal{L}_\varphi^n(1)(x)$$ so that
$\log\lambda\geq \overline{P_x}(T,\varphi),$ and further we have
$\log\lambda\geq\gamma$.

Since $\widehat{X}$ is compact, there exist finitely many points
$x_i\in X (i=1,2,...,q)$ such that
$\widehat{X}=\cup_{i=1}^qB(x_i,\delta)$. Thus
\begin{eqnarray*}\lambda^n&\leq &\sum_i\mu(B(x_i,\delta))\sup_{x\in
B_X(x_i,\delta)}\mathcal{L}_\varphi^n(1)(x)\\
&\leq &e^{K_n}\sum_i
\mu(B(x_i,\delta))\mathcal{L}_\varphi^n(1)(x_i)\\
&\leq
&e^{K_n}q\max_i\mu(B(x_i,\delta))\max_i\mathcal{L}_\varphi^n(1)(x_i),\end{eqnarray*}
where $i$ is such that $\mu(B(x_i,\delta))>0$, and we have
immediately $\log\lambda\leq\gamma.$ Thus $\log\lambda=\gamma.$ \qed

\begin{lem}\label{add11} \ Let $T, \varphi,\lambda$ and $\mu$ be as
in Theorem \ref{thm2.1}. If (\ref{1add}) holds, then we have
(\ref{add10}).
\end{lem}

\begin{proof}\ \ Since $T^n$ is injective on
$T^{-n}_x(B(T^n(x),\delta))$, we have
\begin{eqnarray*}\mu(B(T^n(x),\delta))&=&\int_{T^{-n}_x(B(T^n(x),\delta))}e^{-S_n\varphi(u)+nP(T,\varphi)}d\mu(u)\\
&\leq&e^{K_n-S_n\varphi(x)+nP(T,\varphi)}\mu(T^{-n}_x(B(T^n(x),\delta))).\end{eqnarray*}
Since $\widehat{X}$ is compact, we have $\widehat{X}=\cup_{k=1}^q
B(x_k,\delta/2)$ for $x_k\in X (1\leq k\leq q<\infty)$. Set
$C=\min\{\mu(B(x_k,\delta/2)): 1\leq k\leq q\}>0.$ For $x\in
T^{-n}(X)$, we have $T^n(x)\in B(x_k,\delta/2)$ for some $k$ and
$B(x_k,\delta/2)\subset B(T^n(x),\delta)$. Thus we deduce the left
inequality of (\ref{add10}). The right inequality follows
immediately from the implication.
\begin{eqnarray*}1\geq\mu(B(T^n(x),\delta))&=&\int_{T^{-n}_x(B(T^n(x),\delta))}e^{-S_n\varphi(u)+nP(T,\varphi)}d\mu(u)\\
&\geq&e^{-K_n-S_n\varphi(x)+nP(T,\varphi)}\mu(T^{-n}_x(B(T^n(x),\delta))).\end{eqnarray*}
\end{proof}



In what follows, we discuss possibility of the existence of
invariant probability measures for $T$ on $X$. First we consider
under what condition $\mu\in\mathcal{M}(\widehat{X})$ becomes an
element of $\mathcal{M}(X)$.

\begin{lem}\label{thm2.2} Let $T$, $\varphi$, $\lambda$ and $\mu$ be
as in Theorem \ref{thm2.1}. Assume that $\mu(X\setminus
T^{-n}(X))=0$ for each $n$ and $T$ is backward dense. Then
$\mu(\partial X)=0$, that is, $\mu$ is a possibility measure on $X$.
\end{lem}

Lemma \ref{thm2.2} is extracted from the proof of the result (2) in
Lemma 9 in \cite{Walters}. Walters proved Lemma \ref{thm2.2} under
his condition (1c*), while his method is available to produce our
Lemma \ref{thm2.2}.

The following is essentially Theorem 10 of Walters \cite{Walters}
(See Ledrappier \cite{Ledrappier} for the case when
$X_0=X=\widehat{X}$).

\begin{lem}\label{thm2.3} Let $(T,\psi)$ be admissible with $\mathcal{L}_\psi(1)(x)\equiv 1, \forall\ x\in X$. Then for
$\mu\in \mathcal{M}(\Omega)$ for $\Omega=\widehat{X}$ or $X$, the
following are equivalent:

(1) $\mathcal{L}^*_\psi(\mu)=\mu;$

(2) $\mu\in \mathcal{M}(\Omega,T)$ and for each $f\in
\mathcal{C}(\widehat{X})$,
$$E_\mu(f|T^{-1}\mathcal{B})=\mathcal{L}_\psi(f)\circ
T,\ \mu-a.e;$$

(3) $\mu\in \mathcal{M}(\Omega,T)$ and for each $\nu\in
\mathcal{M}(\Omega,T)$,
$$0=\mu(I_\mu(\mathcal{B} |T^{-1}\mathcal{B})+\psi)\geq \nu(I_\nu(\mathcal{B}
|T^{-1}\mathcal{B})+\psi).$$
\end{lem}

Actually, in terms of Theorem \ref{thm2.1}, Lemma \ref{thm2.3}
asserts the existence of a measure $\mu$ satisfying (1) in Lemma
\ref{thm2.3} over $\widehat{X}$ for $(T,\psi)$. Indeed, since
$\mathcal{L}_\psi(1)(x)\equiv 1$, we have the eigenvalue
$\lambda=1$. Therefore we have

\begin{thm}\label{8add}\ Under the assumption of Lemma\ref{thm2.3}, there exists a
$\mu\in \mathcal{M}(\widehat{X})$ such that (1), (2) and (3) stated
in Lemma \ref{thm2.3} hold.\end{thm}

We shall give a general result in Theorem \ref{thm2.4}. Under the
assumption (1e), we can establish the following result, which was
used in \cite{Walters} to establish some important results.

\begin{lem}\label{lem2.3} Let $(T,\varphi)$ be admissible and $T$
be backward dense. Assume that for some $\delta_0<\delta$,
\begin{equation}\label{2.3+}C_\varphi^{(1)}=\sup_{x\in X}\sup_{T(y)=x}\{|\varphi(y)-\varphi(y')|:
d(x,x')<\delta_0\}<\infty\end{equation} or $X=\widehat{X}$. Then
$\forall \varepsilon>0, \exists N>0$ and $a\in \mathbb{R}$ such that
$\forall x, w\in X$, $\exists y\in T^{-N}x\cap B(w,\varepsilon)$
with $S_N\varphi(y)\geq a.$
\end{lem}

{\bf Proof.}\ \ In terms of (1e), choose $N$ such that $T^{-N}(x)$
is $\varepsilon/4$-dense in $X$ for each $x\in X$. Choose a finite
number of points $w_j (j=1,2,...,s)$ such that
$\widehat{X}=\cup_{j=1}^sB(w_j,\varepsilon/2)$ and for the fixed
$N$, choose finitely many $x_i (i=1,2,...,m)$ such that
$\widehat{X}=\cup_{i=1}^mB(x_i,\tau)$ for some small $\tau$ which is
determined to have ${\rm diam}(B_j^{(k)}(x_i))\break <\varepsilon/4$
for each injective component $B_j^{(k)}(x_i)$ of $T^{-k} (1\leq
k\leq N)$ over $B(x_i,\tau)$. The existence of $\tau$ is confirmed
by the condition (1c).

Let any pair $x, w\in X$ be given. Then $w\in B(w_j,\varepsilon/2)$
for some fixed $j$ and $x\in B(x_i,\tau)$ for some fixed $i$. Since
$T^{-N}(x_i)$ is $\varepsilon/4$-dense in $X$, we can choose a point
$y_i^{(j)}\in T^{-N}(x_i)\cap B(w_j,\varepsilon/4)$. Set
$y=T^{-N}_{y_i^{(j)}}(x)$ and so
$d(T^{N-k}(y),T^{N-k}(y_i^{(j)}))<\varepsilon/4\ (1\leq k\leq N)$
and further
$$d(y,w)\leq d(y,w_j)+d(w_j,w)\leq
d(y,y_i^{(j)})+d(y_i^{(j)},w_j)+\varepsilon/2<\varepsilon,$$ that
is, $y\in T^{-N}x\cap B(w,\varepsilon)$. And we have
\begin{eqnarray*}
S_N\varphi(y)&=&S_N\varphi(y)-S_N\varphi(y_i^{(j)})+S_N\varphi(y_i^{(j)})\\
&\geq & S_N\varphi(y_i^{(j)})-NC_\varphi^{(1)}\\
&\geq&\min_i\{S_N\varphi(y_i^{(j)})\}-NC_\varphi^{(1)}=a_j.\end{eqnarray*}
Put $a=\min_j\{a_j\}$ and then we attain the desired result.\qed

Here we stress that in Lemma \ref{lem2.3} we do not assume any
expanding property for $T$. Walters proved the result in terms of
(1c*), while we observe that actually the condition (1c*) can be
replaced by (1c).

Therefore, we have the following

\begin{thm}\label{thm2.4}\ \ Let $\varphi\in \mathcal{C}(X_0)$ be summable and $(T,\psi)$ be admissible for
$\psi=\varphi-\log\mathcal{L}_\varphi(1)\circ T$. Then there exists
a $\mu\in \mathcal{M}(\widehat{X},T)$ such that
$\mathcal{L}_\varphi(1)\circ T\ e^{-\varphi}$ is the Jacobian of $T$
with respect to $\mu$. Furthermore, if $T$ satisfies (1e), then
$\mu\in\mathcal{M}(X,T)$ and if $\{\mathcal{L}_\psi^n(f):n\geq 0\}$
is equicontinuous for a $f\in \mathcal{C}(\widehat{X})$ and
(\ref{2.3+}) holds, then $\mathcal{L}_\psi^n(f)\rightrightarrows
\mu(f)$ for the $f\in \mathcal{C}(\widehat{X})$.
\end{thm}

{\bf Proof.}\ \ It is obvious that $\mathcal{L}_\psi(1)(x)\equiv 1$
and in terms of Theorem \ref{thm2.1}, there exists a
$\mu\in\mathcal{M}(\widehat{X})$ such that
$\lambda=\mathcal{L}^*_\psi(\mu)(1)=\mu(\mathcal{L}_\psi(1))=1,$
$\mathcal{L}^*_\psi(\mu)=\mu$ and $\mathcal{L}_\varphi(1)\circ T\
e^{-\varphi}$ is the Jacobian of $T$ with respect to $\mu$. Then it
follows from Lemma \ref{thm2.3} that $\mu\in
\mathcal{M}(\widehat{X},T)$. The first part of Theorem \ref{thm2.4}
is proved.

Noting $X\setminus T^{-n}X=\cup_{j=0}^{n-1} T^{-j}(X\setminus
T^{-1}X)$, we have
\begin{eqnarray*} \mu(X\setminus T^{-n}X)&\leq
&\sum_{j=0}^{n-1}\mu(T^{-j}(X\setminus
T^{-1}X))=\sum_{j=0}^{n-1}\mu(X\setminus T^{-1}X)\\
&=&\sum_{j=0}^{n-1}(\mu(X)-\mu(T^{-1}X))=0.\end{eqnarray*} Employing
Lemma \ref{thm2.2} implies that if $T$ satisfies (1e), then
$\mu\in\mathcal{M}(X)$ and so $\mu\in\mathcal{M}(X,T)$.

Now assume that $\{\mathcal{L}_\psi^n(f):n\geq 0\}$ is
equicontinuous for a $f\in \mathcal{C}(\widehat{X})$ and then its
closure is compact in $\mathcal{C}(\widehat{X})$. For any convergent
sequence $\mathcal{L}_\psi^{n_k}(f)$ under the norm, in terms of
Lemma \ref{lem2.3}, the argument in the proof of Theorem 6 in
\cite{Walters} implies that
$\mathcal{L}_\psi^{n_k}(f)\rightrightarrows c\in\mathbb{R}$ as
$k\rightarrow\infty$. Since
$\mu(\mathcal{L}_\psi^{n_k}(f))={\mathcal{L}^*}_\psi^{n_k}(\mu)(f)=\mu(f)$,
we have $c=\mu(c)=\mu(f)$ and so the final part of Theorem
\ref{thm2.4} is proved.\qed

In Theorem \ref{thm2.1}, when $\lambda=1$, $\mu_\varphi$ is a
$e^{-\varphi}$-conformal measure for $T$, but generally,
$\mu_\varphi$ may not be invariant for $T$, and even $\mu_\varphi$
may not be equivalent with $\mu_\psi$ in Theorem \ref{thm2.4}. This
leads us to pose a question.

\begin{que}\label{2.1} Under what condition, are $\mu_\varphi$ and $\mu_\psi$
equivalent?\end{que}

We remark on the condition (1d) for $\psi$ in Theorem \ref{thm2.4},
that is
$$\sum_{T(y)=x}\left|{e^{\varphi(y)}\over
\mathcal{L}_\varphi(1)(x)}-{e^{\varphi(y')}\over
\mathcal{L}_\varphi(1)(x')}\right|\rightarrow 0,\ {\rm as}\
d(x,x')\rightarrow 0.$$ We estimate the quantity in the left side of
above formula:
\begin{eqnarray*}&\ &\sum_{T(y)=x}\left|{e^{\varphi(y)}\over
\mathcal{L}_\varphi(1)(x)}-{e^{\varphi(y')}\over
\mathcal{L}_\varphi(1)(x')}\right|\\
&=&{1\over
\mathcal{L}_\varphi(1)(x)\mathcal{L}_\varphi(1)(x')}\sum_{T(y)=x}
\left|\mathcal{L}_\varphi(1)(x')e^{\varphi(y)}-\mathcal{L}_\varphi(1)(x)e^{\varphi(y')}\right|\\
&\leq &{1\over
\mathcal{L}_\varphi(1)(x)\mathcal{L}_\varphi(1)(x')}\sum_{T(y)=x}
\mathcal{L}_\varphi(1)(x')\left|e^{\varphi(y)}-e^{\varphi(y')}\right|\\
&+&{1\over
\mathcal{L}_\varphi(1)(x)\mathcal{L}_\varphi(1)(x')}\sum_{T(y)=x}
|\mathcal{L}_\varphi(1)(x')-\mathcal{L}_\varphi(1)(x)|e^{\varphi(y')}\\
&=&{1\over
\mathcal{L}_\varphi(1)(x)}\sum_{T(y)=x}\left|e^{\varphi(y)}-e^{\varphi(y')}\right|+{1\over
\mathcal{L}_\varphi(1)(x)}|\mathcal{L}_\varphi(1)(x')-\mathcal{L}_\varphi(1)(x)|\\
&\leq &{2\over
\mathcal{L}_\varphi(1)(x)}\sum_{T(y)=x}\left|e^{\varphi(y)}-e^{\varphi(y')}\right|.
\end{eqnarray*}
Thus if $\inf\{\mathcal{L}_\varphi(1)(x): \forall x\in X\}>0$, then
the condition (1d) for $\varphi$ implies (1d) for $\psi$.

\begin{lem}\label{lem2.4}  Let $(T,\varphi)$ be admissible. Assume that (\ref{2.3+}) holds
or $X=\widehat{X}$. Then $(T,\psi)$ is admissible.\end{lem}

{\bf Proof.}\ \ We only prove Lemma \ref{lem2.4} for the case when
$X=\widehat{X}$. It is obvious that for each $x\in X$, we can find a
$0<\delta_x<\delta$ such that $|\varphi(y')-\varphi(y)|\leq 1$ for
$y\in T^{-1}(x)\cap A_j(x)$ for some fixed $j$ and
$y'=T^{-1}_y(x'),\ \forall \ x'\in B(x,\delta_x).$ Since $X$ is
compact, we can find finitely many points $x_i (1\leq i\leq
M<\infty)$ such that $X=\cup_{i=1}^M B(x_i,\delta_{x_i}).$ Then for
any point $x\in X$, $x\in B(x_i,\delta_{x_i})$ for some $i$ and we
have
$$\varphi(y)=\varphi(y)-\varphi(y_i)+\varphi(y_i)\geq \varphi(y_i)-1=a_i (\rm say),$$
$y_i\in T^{-1}(x_i)\cap A_j(x_i)$ where $j$ is determined as above
and $y=T^{-1}_{y_i}(x)$. Put $a=\min\{a_i: 1\leq i\leq M\}$ and then
$\mathcal{L}_\varphi(1)(x)\geq e^a.$ According to the discussion
before Lemma \ref{lem2.4}, we complete the proof of Lemma
\ref{lem2.4}. \qed

\begin{thm}\label{thm2.6+} Let $T, \varphi, \mu$ and $\lambda$ be as
in Theorem \ref{thm2.1}. Assume that
\begin{equation}\label{2.5+}\lambda^{-n}\mathcal{L}_\varphi^n(g)\rightrightarrows h,\ {\rm
as}\ n\rightarrow\infty\end{equation} for a $g\in
\mathcal{C}(\widehat{X})$ with $g\geq 0$, $\mu(g)=1$ and a $h\in
\mathcal{C}(\widehat{X})$ with $h(x)>0, x\in \widehat{X}$. Then
$m=h\cdot \mu$ is an invariant measure and $\mu(h)=1$ and
$\mathcal{L}_\varphi(h)=\lambda h$.\end{thm}

{\bf Proof.}\ \ It is obvious that
$\mu(\lambda^{-n}\mathcal{L}_\varphi^n(g))=1$ for each $n$, and so
$\mu(h)=1.$
$\mathcal{L}_\varphi(\lambda^{-n}\mathcal{L}_\varphi^n(g))=\lambda^{-n}\mathcal{L}_\varphi^{n+1}(g)$
converges $\mathcal{L}_\varphi(h)$ and $\lambda h$ and hence
$\mathcal{L}_\varphi(h)=\lambda h$.

Set
$$\psi=\varphi-\log\lambda+\log
h-\log h\circ T.$$ We first of all establish the fundamental
equation: $\forall\ f\in\mathcal{C}(\widehat{X})$,
\begin{eqnarray}\label{1} \mathcal{L}_\psi^n(f)(x)&=&\sum_{T^n(y)=x}f(y)\exp
S_n\psi(y)\nonumber\\
&=&{1\over\lambda^nh(x)}\sum_{T^n(y)=x}h(y)f(y)\exp
S_n\varphi(y)\nonumber\\
&=&{1\over\lambda^nh(x)}\mathcal{L}^n_\varphi(hf)(x).\end{eqnarray}
Specially, $\mathcal{L}_\psi(1)(x)=(\lambda
h(x))^{-1}\mathcal{L}_\varphi(h)(x)\equiv 1$ and
$h\mathcal{L}_\psi(f)=\lambda^{-1}\mathcal{L}_\varphi(hf).$ Now we
show that $(T,\psi)$ is admissible. It suffices to check (1d) for
$\psi$. Since $h\in \mathcal{C}(\widehat{X})$, in terms of the
admissible property of $(T,\varphi)$ we have
$$\sum_{T(y)=x}\left|h(y)e^{\varphi(y)}-h(y')e^{\varphi(y')}\right|\rightarrow
0,\ {\rm as}\ d(x,x')\rightarrow 0.$$ By a simple calculation, we
have
$$\sum_{T(y)=x}\left|e^{\psi(y)}-e^{\psi(y')}\right|=\sum_{T(y)=x}\left|{h(y)e^{\varphi(y)}\over \lambda h(x)}-
{h(y')e^{\varphi(y')}\over \lambda h(x')}\right|$$
$$\leq\frac{1}{\lambda}\frac{1}{h(x)}\sum_{T(y)=x}\left|h(y)e^{\varphi(y)}-
h(y')e^{\varphi(y')}\right|+\frac{1}{\lambda}\frac{|h(x)-h(x')|}{h(x)h(x')}\sum_{T(y)=x}h(y')e^{\varphi(y')}$$
$$=\frac{1}{\lambda}\frac{1}{h(x)}\sum_{T(y)=x}\left|h(y)e^{\varphi(y)}-
h(y')e^{\varphi(y')}\right|+\frac{1}{\lambda}\frac{|\mathcal{L}_\varphi(h)(x)-\mathcal{L}_\varphi(h)(x')|}{h(x)}$$
$$\leq\frac{2}{\lambda}\frac{1}{a}\sum_{T(y)=x}\left|h(y)e^{\varphi(y)}-h(y')e^{\varphi(y')}\right|,$$
where $a=\min\{h(x): x\in\widehat{X}\}>0$, and this yields that
$(T,\psi)$ is admissible.

To prove the invariance of the measure $m$, in terms of Lemma
\ref{thm2.3} we only prove the equation $\mathcal{L}^*_\psi(m)=m$.
Actually, for $f\in \mathcal{C}(\widehat{X})$ we have
$$\mathcal{L}^*_\psi(m)(f)=m(\mathcal{L}_\psi(f))=\mu(h\mathcal{L}_\psi(f))=
\mu(\lambda^{-1}\mathcal{L}_\varphi(hf))$$
$$=\lambda^{-1}\mathcal{L}^*_\varphi(\mu)(hf)=\mu(hf)=m(f).$$
Thus we complete the proof of Theorem \ref{thm2.6+}.\qed

Therefore, the crucial point to look for an invariant measure which
is equivalent to $\mu$ is (\ref{2.5+}), that is, uniform convergence
of $\{\lambda^{-n}\mathcal{L}_\varphi^n(g)\}$ for some
$g\in\mathcal{C}(\widehat{X})$ with $\mu(g)=1$. However, we do not
know if the equicontinuity of $\{\mathcal{L}_\psi^n(g)\}$ with
$\psi=\varphi-\log\lambda$ implies uniform convergence of
$\{\mathcal{L}_\psi^n(g)\}$. Obviously, the limit function $h$ is an
element of $\mathcal{C}(\widehat{X})$. We consider the conditions
under which $h(x)>0, x\in \widehat{X}$.

\

(1f) \textsl{$\{T^n\}$ has equivalently uniformly covering property:
there exists a $\delta>0$ such that for each $x\in X$ and each
$n\in\mathbb{N},$ $T^{-n}(B_X(x,\delta))$ can be written uniquely as
a disjoint union of a finite or countable number of open subsets
$A^{(n)}_i(x)\ (1\leq i\leq N_n\leq \infty)$ of $X_0$ and for each
$i$, $T^n$ is a homeomorphism of $A^{(n)}_i(x)$ onto
$B_X(x,\delta)$.}

\

(1g) \textsl{There exists a positive number $C_\varphi$ such that
\begin{equation}\label{2.3}
C_\varphi(x,x')=\sup_{n\geq
1}\sup_{T^n(y)=x}\left|S_n\varphi(y)-S_n\varphi(y')\right|\leq
C_\varphi\end{equation} whenever $d(x,x')<\delta$ for arbitrary pair
$x$ and $x'$ in $X$ and $C_\varphi(x,x')\rightarrow 0$ as
$d(x,x')\rightarrow 0$.}

\

The pair $(T,\varphi)$ is called dynamically admissible if
$(T,\varphi)$ is admissible and satisfies (1f) and (1g).

\begin{lem}\label{lem2.5}\ \ Let all assumptions of Theorem \ref{thm2.6+} with $g(x)>0, x\in
\widehat{X}$, but "$h(x)>0, x\in \widehat{X}$" hold. Assume that
(1g) holds and for each $x\in X$, $\cup_{n=0}^\infty T^{-n}(x)$ is
dense in $X$. Then $h(x)>0, x\in \widehat{X}$.
\end{lem}

{\bf Proof.}\ \ Suppose that for an point $x\in X, h(x)=0$. Since
$\mathcal{L}^n_\varphi(h)(x)=\lambda^nh(x)=0$, we have $h(y)=0,
\forall y\in T^{-n}(x)$ and further, $h(y)=0$ on a dense subset of
$X$. This implies that $h(y)\equiv 0$ on $\widehat{X}$, which
contracts $\mu(h)=1.$ It is obvious that for any pair $x$ and $x'$
in $X$ with $d(x,x')<\delta$, in terms of (\ref{2.3}) we have
$$\lambda^{-n}\mathcal{L}_\varphi^n(g)(x')\leq
Me^{C_\varphi}\lambda^{-n}\mathcal{L}_\varphi^n(g)(x),$$ where $M$
is a constant satisfying $g(y)\leq Mg(y')$, whose existence is
confirmed by the condition "$g(x)>0, x\in\widehat{X}$", so that
$h(x')\leq Me^{C_\varphi}h(x).$ Now suppose that $h(x)=0$ for a
point $x\in\widehat{X}\setminus X$. Take a point $x'\in X$ with
$d(x,x')<\delta/2$ and a sequence $\{x_n\}$ in $X$ such that
$d(x_n,x)\rightarrow 0$ as $n\rightarrow \infty$. For all large $n$,
$d(x_n,x')<\delta$, and thus $h(x')\leq
e^{C_\varphi}h(x_n)\rightarrow 0$ as $n\rightarrow \infty$, and so
$h(x')=0$, a contradiction will be derived as above.\qed

Up to now we have not yet used the expanding property for $T$, that
is, (1c*) in the results we have previously attained. However, we
need the condition (1c*) to confirm the existence of the function
$h$ in Theorem \ref{thm2.6+} and so of the invariant measure, which
was proved by Walters in \cite{Walters}. We remark on (1c*), (1f)
and (1g). It is clear that (1f) follows directly from (1b) and
(1c*), and (1g) implies (1d). The conditions (1a), (1b), (1c*), (1e)
and (1g) are exactly those listed in Walters \cite{Walters}. The
following is Walters' main result.

\begin{thm}\label{thm2.6}  Let the pair
$(T,\varphi)$ be dynamically admissible and $T$ satisfy (1e) and
(1c*). Let $\mu$ and $\lambda$ be as in Theorem \ref{thm2.1}. Then

(1) the pair $(\lambda,\mu)$ is uniquely determined by the
conditions $\lambda>0, \mu\in\mathcal{M}(X)$ and
$\mathcal{L}_\varphi^*(\mu)=\lambda\mu$;

(2) there exists a $h\in \mathcal{C}(\widehat{X})$ with $h>0$ such
that $\mu(h)=1,\ \mathcal{L}_\varphi(h)=\lambda h$;

(3) $h$ satisfies $h(x)\leq e^{C_\varphi(x,x')}h(x')$ and $h$ is
uniquely determined by this condition and the properties $h>0,
\mu(h)=1$ and $\mathcal{L}_\varphi(h)=\lambda h$;

(4) $\lambda^{-n}\mathcal{L}_\varphi^n(f)\rightrightarrows h\cdot
\mu(f), \forall f\in \mathcal{C}(\widehat{X})$;

(5) $m=h\mu$ is a Gibbs invariant measure for $T$ and
$\mathcal{L}^*_\psi(m)=m$, where $$\psi=\varphi-\log\lambda+\log
h-\log h\circ T.$$

(6)
$\log\lambda=\sup\{\nu(I_\nu(\mathcal{B}|T^{-1}\mathcal{B})+\varphi):\nu\in\mathcal{M}(X,T)\}$
and $m$ is the equilibrium state.

(7) $m$ and $\mu$ are positive on nonempty open sets and have no
atoms.
\end{thm}

{\bf Proof.}\ \ For the completeness we state the proof of Theorem
\ref{thm2.6}. It suffices to prove (2), (4) and (5). Consider a
subspace $\Lambda$ of $\mathcal{C}(\widehat{X})$: for a fixed
positive number $\delta_0<\delta$,
$$\Lambda=\{f\in \mathcal{C}(\widehat{X}):\ f\geq 0,\ \mu(f)=1\ {\rm
and}\ f(x)\leq e^{C_\varphi(x,x')}f(x')$$
$${\rm if}\ x,
x' \in X\ {\rm and}\ d(x,x')<\delta_0\}.$$ The argument in the proof
of Theorem 8 of \cite{Walters} implies that $\Lambda$ is nonempty,
convex, closed, bounded and equicontinuous.

Now we want to prove that $\lambda^{-1}\mathcal{L}_\varphi$ is a
linear operator from $\Lambda$ onto $\Lambda$. For any
$f\in\Lambda$, it is easy to see that
$\lambda^{-1}\mathcal{L}_\varphi(f)\geq 0,
\mu(\lambda^{-1}\mathcal{L}_\varphi(f))=\mu(f)=1$. In terms of
(1c*), we have that for $x, x'\in X, d(x,x')<\delta_0$, we have
$d(y,y')<\delta_0$, where $y\in T^{-1}(x)$ and $y'=T^{-1}_y(x')$ and
therefore $f(y)\leq e^{C_\varphi(y,y')}f(y')$. Thus
\begin{eqnarray*}\lambda^{-1}\mathcal{L}_\varphi(f)(x)&=&\lambda^{-1}\sum_{T(y)=x}f(y)e^{\varphi(y)}\\
&\leq&\lambda^{-1}\sum_{T(y)=x}f(y')e^{C_\varphi(y,y')+\varphi(y)}\\
&\leq&\lambda^{-1}\sum_{T(y)=x}f(y')e^{\varphi(y')}e^{C_\varphi(y,y')+\varphi(y)-\varphi(y')}\end{eqnarray*}
\begin{eqnarray*}&\leq&e^{C_\varphi(x,x')}\lambda^{-1}\sum_{T(y')=x'}f(y')e^{\varphi(y')}\\
&\leq&e^{C_\varphi(x,x')}\lambda^{-1}\mathcal{L}_\varphi(f)(x').\end{eqnarray*}
Thus $\lambda^{-1}\mathcal{L}_\varphi(f)\in\Lambda$. Applying the
Schauder-Tychonoff fixed-point theorem yields that
$\lambda^{-1}\mathcal{L}_\varphi$ has a fixed point $h\in\Lambda$.
The property $h>0$ follows from Lemma \ref{lem2.5}. Therefore, (2)
has been proved.

To prove (4) and (5). Notice the expression of $\psi$. As in the
proof of Theorem 6 of \cite{Walters}, we can show that for any $f\in
\mathcal{C}(\widehat{X})$, $\{\mathcal{L}_\psi^n(f)\}$ is
equicontinuous. Actually, we have
\begin{eqnarray*}
&\ &|S_n\psi(y)-S_n\psi(y')|\leq|S_n\varphi(y)-S_n\varphi(y')|\\
&+&|\log h(y)-\log h(y')|+|\log h(x)-\log h(x')|\\
&\leq&C_\varphi(x,x')+\frac{2}{a}\sup\{|h(u)-h(v)|:\ d(u,v)\leq
d(x,x')\},\end{eqnarray*} where $a=\min\{h(x): x\in\widehat{X}\}$,
and hence $C_\psi(x,x')\rightarrow 0$ as $d(x,x')\rightarrow 0$.

Applying Theorem \ref{thm2.4} to $\psi$ instead of $\varphi$ yields
the existence of $m\in \mathcal{M}(X,T)$ with
$\mathcal{L}^*_\psi(m)=m$ and
$\mathcal{L}_\psi^n(f)\rightrightarrows m(f)$ as
$n\rightarrow\infty$. Since from (\ref{1})
$\mathcal{L}_\psi^n(f)=h^{-1}\lambda^{-n}\mathcal{L}_\varphi^n(hf)$,
we have $\lambda^{-n}\mathcal{L}_\varphi^n(hf)\rightrightarrows
h\cdot m(f)$ and so
$\lambda^{-n}\mathcal{L}_\varphi^n(f)\rightrightarrows h\cdot
m(f/h)$. Furthermore
$\mu(f)=\mu(\lambda^{-n}\mathcal{L}_\varphi^n(f))\rightrightarrows\mu(h\cdot
m(f/h))=m(f/h)$ and equivalently $m=h\cdot\mu.$ We have proved (4)
and (5). \qed

Let us remark on the subspace $\Lambda$ of
$\mathcal{C}(\widehat{X})$. For the fixed $\delta_0$, the
boundedness of $\Lambda$ can be proved without the condition (1c*),
while in the proof of that $\lambda^{-1}\mathcal{L}_\varphi$ becomes
a linear operator from $\Lambda$ onto $\Lambda$, the condition (1c*)
cannot be avoided. If we change the definition of $\Lambda$ with
$\delta_0$ replaced by a positive number $\delta(f)$ depending on
$f$, then we do not need (1c*) to prove that
$\lambda^{-1}\mathcal{L}_\varphi$ becomes a linear operator from
$\Lambda$ onto $\Lambda$, while the boundedness of $\Lambda$ cannot
be proved.

Now we complete the proof of  Theorem \ref{thm1.1}. First let us
recall Theorem \ref{thm1.1} says that

{\sl Let the pair $(T,\varphi)$ be admissible and for some fixed
$N\in\mathbb{N}$, $T^N$ satisfy (1c*) and (1g) for some $\delta_N$
and (1e). Then all the statements listed in Theorem \ref{thm2.6}
still hold.}

{\bf Proof of Theorem \ref{thm1.1}}. Since $(T,\varphi)$ be
admissible, in terms of Theorem \ref{thm2.1} the linear operator
$\mathcal{L}_\varphi$ of $\mathcal{C}(\widehat{X})$ to itself exists
and the corresponding $\mu$ and $\lambda$ exist. And in view of
Lemma \ref{lem2.2}, $(T^N,S_N\varphi)$ is admissible, and (1b) and
(1c*) for $T^N$ and for some $\delta'_N\leq \delta_N$ imply (1f) for
$T^N$ and $\delta'_N$. Thus $(T^N,S_N\varphi)$ is dynamically
admissible.

It suffices to prove (2) and (4) in Theorem \ref{thm2.6}. Since
$\mathcal{L}_\varphi^*(\mu)=\lambda\mu$, we have
$$\mathcal{L}^*_{S_N\varphi,T^N}(\mu)={\mathcal{L}_\varphi^*}^N(\mu)=\lambda^N\mu.$$
In terms of Theorem \ref{thm2.6}, there exists a $h\in
C(\widehat{X})$ with $h>0$ such that $\mu(h)=1,\
\mathcal{L}^N_\varphi(h)=\mathcal{L}_{S_N\varphi,T^N}(h)=\lambda^N
h$ and for each $f\in\mathcal{C}(\widehat{X})$
$$\lambda^{-nN}\mathcal{L}_\varphi^{Nn}(f)\rightrightarrows
h\cdot\mu(f),\ {\rm as}\ n\rightarrow\infty.$$ Thus as
$n\rightarrow\infty$, we have
$$\lambda^{-nN}\mathcal{L}_\varphi^{nN+1}(f)=\mathcal{L}_\varphi(\lambda^{-nN}\mathcal{L}_\varphi^{Nn}(f))
\rightrightarrows
\mathcal{L}_\varphi(h\cdot\mu(f))=\mu(f)\mathcal{L}_\varphi(h)$$ and
$$\lambda^{-nN}\mathcal{L}_\varphi^{Nn+1}(f)=\lambda^{-nN}\mathcal{L}_\varphi^{Nn}(\mathcal{L}_\varphi(f))
\rightrightarrows
h\cdot\mu(\mathcal{L}_\varphi(f))=h\cdot\mathcal{L}^*_\varphi\mu(f)=\mu(f)\lambda
h.$$ This implies immediately
$$\mathcal{L}_\varphi(h)=\lambda h,$$ that is, (2) has been proved.

(4) follows from the following implication: for each $0\leq i<N$, we
have
$$\lambda^{-nN-i}\mathcal{L}_\varphi^{nN+i}(f)=\lambda^{-i}\left(
\lambda^{-nN}\mathcal{L}_\varphi^{Nn}(\mathcal{L}_\varphi^i(f))\right)\rightrightarrows
\lambda^{-i}h\cdot\mu(\mathcal{L}_\varphi^i(f))=h\cdot\mu(f).$$ \qed

We remark on the conditions in Theorem \ref{thm1.1}. We cannot
deduce that $(T,\varphi)$ is admissible in terms of the dynamically
admissible property of $(T^N,S_N\varphi)$ and the conditions on
$T^N$ in Theorem \ref{thm1.1} and thus we cannot obtain
$\mathcal{L}_\varphi$, $\mu$ and $\lambda$.

Finally, we mention that the expanding property is not necessary for
the existence of conformal measure, while in the discussion of this
section it is necessary for the existence of an invariant measure
which is equivalent to the conformal measure.

\section{Bowen Formula on Invariant Sets}

As an application of the previous results, in this section, we
establish the Bowen formula on some special subsets of $X_0$ and
discuss the existences of conformal and invariant measures dealing
with the derivatives. Here generally, we do not require the metric
space $(X,d)$ is embedded into a compact metric space, while we
assume that $(X,d)$ is locally compact, that is to say, for each
$x\in X$ and $R>0$, $\overline{B(x,R)}$ is compact.

Define
\begin{equation}\label{1.1} D_dT(x)=\lim_{y\rightarrow
x}{d(T(y),T(x))\over d(y,x)}\end{equation} if the limit exists and
$D_dT(x)$ is called derivative of $T$ at $x$ with respect to the
metric $d$. We say that $T$ has bounded distortion on a subset $U$
of $X_0$ if $D_dT(x)$ exists at each point of $U$ and for some
$M=M(U)>0$, we have
$${D_dT(x)\over D_dT(y)}\leq M$$ for arbitrary pair $x$ and $y$
in $U$ and $M$ is named distortion constant. It is obvious that if
$T$ has the derivative $D_dT(x)$ in $X_0$, then for each $n\in
\mathbb{N}$ and each $x\in T^{-n}X$ we have
\begin{equation}\label{3.2}
D_dT^n(x)=\prod_{k=0}^{n-1}(D_dT)(T^k(x)).\end{equation}

If $X$ is a subset of the Riemann sphere $\widehat{\mathbb{C}}$,
consider a Riemannian metric $\tau:\tau(z)|dz|$. If $f(z)$ is
meromorphic on $X_0$, then the derivative of $f$ with respect to
$\tau$ at $z\in X_0$ is
$$D_\tau f(z)=|f'(z)|{\tau(f(z))\over \tau(z)}$$
and in particular, for $\tau(z)=(1+|z|^t)^{-1}$, we write
$D_tf(z)=D_\tau f(z)$. When $t=2$, $D_2f(z)$ is the derivative of
$f(z)$ with respect to the Riemann sphere metric, usually denoted by
$f^\times(z)$; When $t=0$, $D_0f(z)=|f'(z)|$.

Let $T:X_0\rightarrow X$ have the derivative on $X_0$. Consider the
following Poincar\'e sequence, for $t\geq 0$ and $a\in X$,
$$\mathcal{L}_{t,T}^n(a):=\sum_{T^n(z)=a}D_dT^n(z)^{-t}.$$ Actually,
$\mathcal{L}_{t,T}^n(a)=\mathcal{L}^n_{\varphi,T}(1)(a)$ with
$\varphi=-t\log D_dT(x)$ and for a fixed $m\in\mathbb{N},$
$S_m\varphi(x)=-t\log D_dT^m(x)$. Thus
$\mathcal{L}_{t,T^m}(a)=\mathcal{L}_{S_m\varphi,T^m}(1)(a)=\mathcal{L}^m_{t,T}(a)$.
If the confusion cannot occur, we simply write $\mathcal{L}_t(a)$
for $\mathcal{L}_{t,T}(a)$. And we write the (resp., upper and
lower) pressure of $T$ for $\varphi=-t\log D_dT(x)$ as $P(T,t)$
(resp., $\overline{P_a}(T,t)$ and $\underline{P_a}(T,t)$). If it is
finite, then $P(T,t)$ is a real function in $t$. The Bowen formula
is to reveal the relation between some $t$ and the Hausdorff
dimension of some set.

Following Kotus and Urbanski \cite{KotusUrbanski}, we introduce the
following concept.

\begin{defin}\label{def3.1}\ {\sl $T$ is called weak Walters
expanding (with expanding constant $C\geq 1$), provided that

(2a) $T$ satisfies (1a), that is, the set $T^{-1}(x)$ for each $x\in
X$ is at most countable;

(2b) For each $x\in X$ there exists a $\delta_x>0$ such that for
each $n\in\mathbb{N}$, $T^n$ is a homeomorphism of every component
of $T^{-n}(B(x,\delta_x))$ onto $B(x,\delta_x)$;

(2c) $\forall \varepsilon>0$ and $\forall x\in X$, $\exists
\delta_0$ with $0<\delta_0<\delta_x$ such that for each $y\in X_0$
with $T(y)=x$, once $d(x,x')<\delta_0$ for $x'\in X$, we have
$d(T^{-1}_y(x), T^{-1}_y(x'))<\varepsilon,$ where $T^{-1}_y$ is the
branch of the inverse of $T$ which sends $x$ to $y$;

(2d) For each $a\in X$, there exist $C(a)\geq 1$, $\varrho(a)>0$ and
$N(a)\geq 1$ such that for each $n$
$$d(T^{nN}(x),T^{nN}(y))\geq \varrho(a)C^n(a) d(x,y)$$
whenever $x$ and $y$ lie in a component of $T^{-nN}(B(a,\delta_a))$;

(2e) For an arbitrary point $x\in X_0$ and $\delta>0$, given a
compact subset $K$ of $X$ there exists a positive $M=M(K)$ such that
$K\subseteq T^M(B(x,\delta)\cap T^{-M}(X))$.

$C=\inf\{C(a):a\in X\}$ is called the expanding constant for $T$. If
$(X,d)$ is compact, then $T$ is called a Walters expanding map (with
the expanding constant $C$).}
\end{defin}

When $X$ is embedded into a compact metric space $\widehat{X}$, the
above conditions with $\delta=\inf\{\delta_x:\forall x\in X\}>0$,
$C=1$, $\varrho(a)=1$ and $N=1$ are those which Walters considered
(see Section 2). In this case we note that (2c) follows directly
from (2d) with $N=1$, but the implication is not available for
$N>1$. And (2c) is necessary for $\mathcal{L}_\varphi$ being a
linear operator from $\mathcal{C}(\widehat{X})$ to itself. The
Walters expanding maps were first named by Kotus and Urbanski
\cite{KotusUrbanski} with $C>1$ and with (2e) replaced by (1e) but
without (2c), that is to say, the definition here is different a bit
from the Kotus and Urbanski's. Actually, if $(X,d)$ is compact, then
(2e) is equivalent to (1e) and we can find a fixed $N$ independent
of $a$ in (2d) and if for each $a\in X$, $C(a)>1$, then $C>1$. In
the definition of Kotus and Urbanski with $X=\widehat{X}$, it seems
to allow $N>1$. Actually, we can use the metric $\widehat{d}$
defined by
$$\widehat{d}(x,y)=\sum_{k=0}^{N-1}\widehat{C}^{-k}d(T^k(x),T^k(y)),\
\widehat{C}=\sqrt[N]{C}.$$ It is easy to see that
$$\widehat{d}(T(x),T(y))\geq \widehat{C}\widehat{d}(x,y),$$
whenever $x$ and $y$ lie on a component of $T^{-1}(B(a,\delta)).$ We
seem unclear to understand how one could imply the inequality (1) in
\cite{KotusUrbanski} for $N>1$, for, although we have for $n=1$
$$|\phi(T^{-1}_u(y))-\phi(T^{-1}_u(z))|\leq
Ld^\beta(T^{-1}_u(y),T^{-1}_u(z))$$ in terms of the dynamically
H\"older continuous condition of $\phi$, but we cannot compare
$d(T^{-1}_u(y),T^{-1}_u(z))$ to $d(y,z)$. However using the metric
$\widehat{d}$, instead of $d$, is no problem.

\begin{defin}\label{def3.2}\ {\sl A continuous map $T:X_0\rightarrow X$ is called conformal if
the derivative of $T$ with respect to $d$ exists at each $x\in X_0$
and for each $n\in \mathbb{N}$, each $x\in X_0$ and some
$\delta_x>0$, $T^n$ has bounded distortion in each injective
component $A_j^{(n)}(x)$ of $T^{-n}$ over $B(x,\delta_x)$ with the
distortion constant only depending on $x$, denoted by $M(x)$, and
for arbitrary pair $y, y'\in A_j^{(n)}(x)$, there exists a point
$w\in A_j^{(n)}(x)$ such that
$$d(T^n(y),T^n(y'))\leq D_dT^n(w)d(y,y').$$}\end{defin}

The above inequality for $d$ implies one for $\widehat{d}$.
Obviously, (\ref{1add}) holds for $\varphi=-t\log D_dT(x)$ and
$\delta=\delta_x$ if for each $n\in \mathbb{N}$, each $x\in X_0$ and
some $\delta_x>0$, $T^n$ has uniformly bounded distortion mentioned
in Definition \ref{def3.2}. Therefore, if $T:X_0\rightarrow X$ is
conformal, (\ref{1add}) holds for every $x\in X$ and $\varphi=-t\log
D_dT(x)$ with $\delta=\delta_x$ and $K_n$ depending on $x$. The same
argument as in the proof of Theorem \ref{6add} produces the
following, where $\overline{P_a}(T,t)=\infty$ is allowed.

\begin{lem}\label{lem3.1}\ Let $T:X_0\rightarrow X$ satisfy (2a), (2b) and (2e) and be a conformal map.
Then the following statements hold:

(1) $\overline{P_a}(T,t)$ and $\underline{P_a}(T,t)$ are independent
of $a\in X$ and so we simply write $\overline{P}(T,t)$ and
$\underline{P}(T,t)$, in turn, for $\overline{P_a}(T,t)$ and
$\underline{P_a}(T,t)$;

(2) For a fixed $m$, $m\overline{P}(T,t)=\overline{P}(T^m,t)$ and
$m\underline{P}(T,t)=\underline{P}(T^m,t)$;

(3) If, in addition, $X$ is compact, then
$P(T,t)=\overline{P}(T,t)=\underline{P}(T,t)$.\end{lem}

In terms of Theorem \ref{thm2.1}, we give out conditions under which
there exists a $D_dT^t(x)$-conformal measure, which is simply
written into $t$-conformal measure if no confusion occurs.

\begin{thm}\label{thm3.1+}\ Let $(X,d)$ be compact and
let $T:X_0\rightarrow X$ satisfy (2a), (2b), (2c) and (2e) and be a
conformal map. If $(T,\varphi_t)$ with $\varphi_t=-t\log D_dT(x)$ is
admissible, then there exists a $\mu\in\mathcal{M}(X)$ such that
$\mathcal{L}_t^*(\mu)=\lambda\mu$ and
$\lambda=\mathcal{L}_t^*(\mu)(1)>0$ and further,
$\log\lambda=P(T,t)$, and if $P(T,t)=0$, then $T$ has a
$t$-conformal measure $\mu$ on $X$.\end{thm}

Next we discuss the conditions under which $(T,\varphi_t)$ is
admissible or dynamically admissible.

\begin{lem}\label{lem3.2}\ Let $T:X_0\rightarrow X$ satisfy (2a), (2b) and (2c) and be a conformal map.
Assume that $(X,d)$ is compact and $\varphi_t=-t\log D_dT(x)$ is
summable on $X$. Then the following statements hold.

(1) If
\begin{equation}C^{(1)}(x,x')=\sup_{T(y)=x}\left|1-{D_dT(y)\over
D_dT(y')}\right|\rightarrow 0,\ {\rm as}\ d(x,x')\rightarrow
0,\end{equation} then $(T,\varphi_t)$ is admissible;

(2) If \begin{equation}C(x,x')=\sup_{n\geq
1}\sup_{T^n(y)=x}\left|1-{D_dT^n(y)\over
D_dT^n(y')}\right|\rightarrow 0,\ {\rm as}\ d(x,x')\rightarrow
0,\end{equation} then $(T,\varphi_t)$ is dynamically admissible.
\end{lem}

{\bf Proof.} We can write
$${D_dT(y)\over
D_dT(y')}=1+CC^{(1)}(x,x')$$ with $|C|\leq 1$ and
$$\left({D_dT(y)\over
D_dT(y')}\right)^t=(1+CC^{(1)}(x,x'))^t=1+tC(1+o(1))C^{(1)}(x,x')$$
and so $$C^{(1)}_t(x,x')=\sup_{T(y)=x}\left|1-\left({D_dT(y)\over
D_dT(y')}\right)^t\right|\rightarrow 0,\ {\rm as}\
d(x,x')\rightarrow 0.$$

The admissible property of $(T,\varphi_t)$ follows from the
following implication:
$$\sum_{T(y)=x}\left|e^{\varphi_t(y)}-e^{\varphi_t(y')}\right|=
\sum_{T(y)=x}D_dT(y)^{-t}\left|1-\left({D_dT(y)\over
D_dT(y')}\right)^t\right|$$
$$\leq \mathcal{L}_t(1)(x)C^{(1)}_t(x,x')\leq \sup\{\mathcal{L}_t(1)(x):x\in X\}C_t^{(1)}(x,x')\rightarrow 0,$$
as $d(x,x')\rightarrow 0.$

The dynamically admissible property of $(T,\varphi_t)$ follows from
the following implication: for each $n$,
\begin{eqnarray*}&\ &|S_n\varphi_t(y)-S_n\varphi_t(y')|=t\left|\log{D_dT^n(y)\over
D_dT^n(y')}\right|\\
&\leq& t\left(\left|1-{D_dT^n(y)\over
D_dT^n(y')}\right|+\left|1-{D_dT^n(y')\over
D_dT^n(y)}\right|\right)\\
&=&t\left(1+{D_dT^n(y')\over
D_dT^n(y)}\right)\left|1-{D_dT^n(y)\over D_dT^n(y')}\right|\\
&\leq& t(1+M(x))C(x,x')\rightarrow 0\ (d(x,x')\rightarrow
0).\end{eqnarray*}\qed

Walters in \cite{Walters} and Kotus and Urbanski in
\cite{KotusUrbanski} considered the H\"older continuous condition
for the test function $\varphi$.

\begin{lem}\label{lem3.3}\ Let $T$ be a Walters expanding map with the expanding constant $C>1$ and
let $\varphi_s=-s\log D_dT(x)$ be summable and locally uniformly
H\"older continuous, that is, for $d(x,x')<\epsilon$, we have
$$|\varphi_s(x)-\varphi_s(x')|\leq Ld(x,x')^\sigma,$$ where $L$ and $\sigma$ are two positive constants.  Then
$(T^n,S_n\varphi_s)$ is dynamically admissible on $X$.\end{lem}

{\bf Proof.}\ For $q\leq n$, we write $n=mN+p$ and $q=jN+k$ for some
$0\leq p<N$, $0\leq j\leq m$ and $0\leq k<N$. For each $y\in
T^{-n}(x)$ and $y'=T^{-n}_y(x')$ with $d(x,x')<\delta$, we treat two
cases: when $p\geq k$,
$$d(T^q(y),T^q(y'))=d(T^{-(m-j)N+(k-p)}(x),T^{-(m-j)N+(k-p)}(x'))$$
$$\leq \varrho^{-1}C^{-(m-j)}d(T^{k-p}(x),T^{k-p}(x')),$$ where $T^{k-p}$
denotes a branch over $x$ and $x'$; when $p<k$,
$$d(T^q(y),T^q(y'))\leq \varrho^{-1}C^{-(m-j-1)}d(T^{-N+k-p}(x),T^{-N+k-p}(x')),$$ where
$T^{-N+k-p}$ denotes a branch over $x$ and $x'$. Set
$$C_N(x,x')=\sup_{T^{N}(y)=x}\sum_{k=0}^{N-1}d(T^k(y),T^k(y'))^\sigma.$$
Clearly, $C_N(x,x')\rightarrow 0$ as $d(x,x')\rightarrow 0$ with
help of (2c). Thus we have
\begin{eqnarray*}|S_n\varphi_s(y)-S_n\varphi_s(y')|&\leq&
\sum_{j=0}^{m-1}|S_N\varphi_s(T^{jN}(y))-S_N\varphi_s(T^{jN}(y'))|\\
&+& |S_{p-1}\varphi_s(T^{mN}(y))-S_{p-1}\varphi_s(T^{mN}(y'))|\\
&\leq&\sum_{j=0}^{m-1}L\sum_{k=0}^{N-1}d(T^{jN}(T^k(y)),T^{jN}(T^k(y')))^\sigma\\
&+&L\sum_{k=0}^{p-1}d(T^{mN}(T^k(y)),T^{mN}(T^k(y')))^\sigma\\
&\leq&L\varrho^{-\sigma}\sum_{j=0}^{m}C^{-(m-j-1)\sigma}\sum_{k=0}^{N-1}d(T^{k}(y),T^{k}(y'))^\sigma\\
&\leq& L\varrho^{-\sigma}{C^\sigma\over
C^{\sigma}-1}C_N(x,x').\end{eqnarray*} This completes the proof of
Lemma \ref{3.3}.\qed

We discuss the further property of the pressure $P(T,t)$.

\begin{lem}\label{lem3.4}\ Let $T:X_0\rightarrow X$ be a weak Walters expanding conformal map
with the expanding constant $C(a)\geq 1$. Then $P(T,t)$ is convex,
non-increasing and so continuous in $t\in(\tau(T),+\infty)$ with
$\tau(T)=\inf\{t\geq 0: P(T,t)<\infty\}$, and if $C(a)>1$, $P(T,t)$
is strictly decreasing in $t\in(\tau(T),+\infty)$.
\end{lem}

{\bf Proof.}\ The convexity of $P(T,t)$ in $t$ is obvious. For a
fixed $a\in X$, from (2) in Lemma \ref{lem3.1}, we only need to
prove that $P(T^N,t)$ with $N=N(a)$ is non-increasing and further
strictly decreasing in $t$ if $C>1$. We write
$$\mathcal{L}^n_{t,T^N}(a)=\sum_{T^{nN}(y)=a}D_dT^{nN}(y)^{-t}.$$
By $S_m(t)$ we denote the sum of $m$ items of the above series.
Clearly, the condition (2d) yields that $D_dT^{nN}(y)\geq \varrho
C^n(a),\ y\in T^{-nN}(a)$. Then
\begin{eqnarray*}{\partial S_m(t)\over\partial t}&=&\sum^*{1\over
D_dT^{nN}(y)^t}\log {1\over D_dT^{nN}(y)}\\
&\leq&-(n(\log C)+\log\varrho)S_m(t),\end{eqnarray*} where
$\sum\limits^*$ is the sum of the items in $S_m(t)$. For a pair
$t_1$ and $t_2$ with $\tau(T)<t_2<t_1$, we have $$\frac{1}{n}\log
S_m(t_1)-\frac{1}{n}\log S_m(t_2)\leq -(\log
C+\frac{1}{n}\log\varrho)(t_1-t_2).$$ For all sufficiently large
$m$, we have $S_m(t_1)\leq \mathcal{L}^n_{t_1,T^N}(a)\leq 2S_m(t_1)$
and thus
$$\frac{1}{n}\log
\mathcal{L}^n_{t_1,T^N}(a)-\frac{1}{n}\log
\mathcal{L}^n_{t_2,T^N}(a)\leq \frac{1}{n}(\log 2-\log\varrho)-(\log
C)(t_1-t_2)$$ so that
$$P(T^N,t_1)-P(T^N,t_2)\leq -(\log C)(t_1-t_2).$$
The proof of Lemma \ref{lem3.4} is completed.
 \qed

Define the number
$$s(T)=\inf\{t\geq 0: P(T, t)\leq 0\}$$
as the Poincar\'e exponent for $T$ (if for all $t$, $P(T,t)>0$, then
define $s(T)=\infty$). We do not know if $\tau(T)<\infty$ and
$s(T)<\infty$ for a weak Walters expanding map. The following result
gives a condition under which $s(T)<+\infty$ and discusses the
relation between $s(T)$ and the Hausdorff dimension of some subset
of $X_0$. Set
$$X_\infty=\bigcap_{n=0}^\infty T^{-n}(X).$$
It is obvious that $X_\infty$ is completely invariant, that is,
$T(x)\in X_\infty$ if and only if $x\in X_\infty$. Define $X_r$ as
the set of points $x$ in $X_\infty$ such that $\{T^n(x)\}$ has a
limit point in $X$ and $X_r$ is called the radial set on $X$ for
$T$. When $(X,d)$ is compact, we have $X_r=X_\infty$.

\begin{thm}\label{thm3.1}\ \ Let $T:X_0\rightarrow X$ be a weak Walters expanding conformal map
with the expanding constant $C(a)>1$ at each point $a\in X$. Then
${\rm dim}_H(X_r)\leq s(T).$ In addition, assume that there exist a
point $x\in X$ and a $R>0$ such that for arbitrary two points $a$
and $b$ in $X\setminus B(x,R)$ and each $n$, we have a single valued
branch $g$ of $T^{-n}$ which has bounded distortion over $a$ and $b$
and there exists a positive function $\phi(r)$ in $(0,\infty)$ such
that for all sufficiently large $\hat{R}$, we have
$\phi(\hat{R})\geq\hat{R}$, and if $g(a)\in B(x,\hat{R})$ then
$g(b)\in B(x,\phi(\hat{R}))$. Then
$${\rm dim}_H(X_r)=s(T).$$
\end{thm}

{\bf Proof.}\ The proof of the first part of Theorem \ref{thm3.1} is
the same as that of Theorem 2.7 of \cite{KotusUrbanski} and Lemma
3.6 in \cite{Zheng}, because we may assume that $s(T)<+\infty$.

The main idea to prove the second part of Theorem \ref{thm3.1} comes
from Stallard \cite{Stallard} and Zheng \cite{Zheng}. Noting that
$s(T^N)=s(T)$, we can assume that $N=1$ in (2d). Take arbitrarily
$t<s(T)$ and so $P(T,t)>0$ so that for a sequence of positive
integers, $\mathcal{L}^n_t(a)\rightarrow \infty$ as
$n\rightarrow\infty$.

Now we want to prove for an arbitrarily large $A>0$, there exist a
sequence of $m\in \mathbb{N}$ such that
\begin{equation}\label{3.3}\sum_{y\in B(a,\delta_a/2),
T^m(y)=a}D_dT^m(y)^{-t}>A\end{equation} for $a\in X$. Take in $X$
points $x_i (1\leq i\leq p)$ such that
$B_X(x,2R)\subset\cup_{i=1}^pB_X(x_i,\delta_i),$
$\delta_i=\delta_{x_i}$ and $x_p\in X\setminus B(x,R)$. Then there
exist a $n$ and a $R_n>\delta_i (1\leq i\leq p)$ such that for each
$i$, \begin{equation}\label{3.4}\sum_{y\in B(x,R_n),
T^n(y)=x_i}D_dT^n(y)^{-t}>AM^{-t}(x_i).\end{equation} For $c\in
B(x_i,\delta_i) (1\leq i\leq p)$ and for each $y_i\in
T^{-n}(x_i)\cap B_X(x,R_n)$, in terms of (2d) we have
$w=T^{-n}_{y_i}(c)$ such that
$$d(w,y_i)\leq \varrho^{-1}C^{-n}d(c,x_i)<\varrho^{-1}C^{-n}\delta_i<R_n.$$
This together with (\ref{3.4}) implies that
\begin{equation}\label{3.5}\sum_{y\in
B(x,\tilde{R}_n), T^n(y)=c}D_dT^n(y)^{-t}\geq\sum_{y\in B(x,2R_n),
T^n(y)=c}D_dT^n(y)^{-t}>A\end{equation} where
$\tilde{R}_n=\phi(2R_n).$ For $c\in B_X(x,\tilde{R}_n)\setminus
B(x,R)$ and for each $y_p\in T^{-n}(x_p)\cap B_X(x,2R_n)$, in terms
of our assumption of Theorem \ref{thm3.1} we have
$w_p=T^{-n}_{y_p}(c)$ and $w_p\in B(x,\tilde{R}_n)$ so that
(\ref{3.5}) holds for such $c$. By induction, for each $s\geq 1$ we
have
\begin{equation}\label{3.6}\sum_{y\in
B(x,\tilde{R}_n), T^{ns}(y)=a}D_dT^{ns}(y)^{-t}>A^s,\ \forall\ a\in
B_X(x,\tilde{R}_n).\end{equation} Take $a_j (1\leq j\leq q)$ in
$B(x,\tilde{R}_n)$ such that $B(x,\tilde{R}_n)\subset \cup_{j=1}^q
B(a_j,\varrho_j/2), \varrho_j=\delta_{a_j}/2$. Take a $s$ such that
$A^s>qAM^t(a_j)\ (1\leq j\leq q)$. We want to prove that for some
$a_j$, (\ref{3.3}) holds. For the sake of simplicity, assume that
$q=2$ and from (\ref{3.6}) assume that
$$\sum_{y\in
B(a_2,\varrho_2/2), T^{ns}(y)=a_1}D_dT^{ns}(y)^{-t}>{A^s\over q},$$
and
$$\sum_{y\in
B(a_1,\varrho_1/2), T^{ns}(y)=a_2}D_dT^{ns}(y)^{-t}>{A^s\over q}.$$
Then \begin{eqnarray*}&\ &\sum_{y\in B(a_1,\varrho_1),
T^{2ns}(y)=a_1}D_dT^{2ns}(y)^{-t}\\
&\geq & \sum_{w\in B(a_2,\varrho_2/2),
T^{ns}(w)=a_1}D_dT^{ns}(w)^{-t} \sum_{y\in
B(a_1,\varrho_1), T^{ns}(y)=w}D_dT^{ns}(y)^{-t}\\
&\geq &\sum_{w\in B(a_2,\varrho_2/2),
T^{ns}(w)=a_1}D_dT^{ns}(w)^{-t} \sum_{y\in
B(a_1,\varrho_1/2), T^{ns}(y)=a_2}D_dT^{ns}(y)^{-t}M^{-t}(a_2)\\
&\geq &\left({A^s\over q}\right)^2M^{-t}(a_2)>A.\end{eqnarray*} Thus
we have proved (\ref{3.3}).

For each $y\in B(a,\delta_a/2)$ with $T^m(y)=a$, we have
$${\rm diam}(T^{-m}_y(B(a,\delta_a)))\leq 2\varrho^{-1}
C^{-m}\delta_a<\delta_a/2$$ so that $T^{-m}_y(B(a,\delta_a))\subset
B(a,\delta_a)$. Set
$$\alpha(y)=\inf\left\{{d(T^{-m}_y(b), T^{-m}_y(c))\over d(b,c)}:\ b,
c\in B(a,\delta_a)\right\}.$$ It is clear from conformal and
expanding properties of $T$ that
$$M^{-1}(a)D_dT^m(y)^{-1}\leq \alpha(y)\leq \varrho^{-1}C^{-m}(a)<1$$ and so
$$\sum_{y\in B(a,\delta_a),
T^{m}(y)=a}\alpha(y)^{-t}\geq \sum_{y\in B(a,\delta_a),
T^{m}(y)=a}D_dT^{m}(y)^{-t}M^{-t}(a)>1.$$ This yields that the
invariant set for the system
$\{T^{-m}_y:B(a,\delta_a)\hookrightarrow B(a,\delta_a)|\ y\in
B(a,\delta_a/2)\cap T^{-m}(a)\}$ has the Hausdorff dimension at
least $t$ and is contained in $X_r$. Furthermore, ${\rm
dim}_H(X_r)\geq t$ and so ${\rm dim}_H(X_r)\geq s(T)$.

The proof of Theorem \ref{thm3.1} is completed.\qed

There exists a direct consequence of Theorem \ref{thm3.1} that
$s(T)$ is finite under the assumption of Theorem \ref{thm3.1}.

\begin{cor}\label{cor3.1}\ Let $T:X_0\rightarrow X$ be a weak Walters expanding conformal
map and satisfy all the assumptions of the second part of Theorem
\ref{thm3.1}. If ${\rm dim}_H(X)<\infty$, then

(1) $s(T)\leq {\rm dim}_H(X)<\infty$;

(2) as $n\rightarrow\infty$, $\mathcal{L}_t^n(a)\rightarrow \infty$
for $t<s(T)$ or $0$ for $t>s(T)$.\end{cor}

Next we consider the case of the Walters expanding conformal map.
The following is a direct consequence of Theorem \ref{thm3.1}.

\begin{cor}\label{cor3.2}\ Let $T:X_0\rightarrow X$ be a Walters expanding conformal
map with expanding constant $C>1$. Then
\begin{equation}\label{3.7}s(T)={\rm dim}_H(X_r)={\rm
dim}_H(X_\infty).\end{equation}\end{cor}

{\bf Proof.} Since $(T,d)$ is compact, for a point $x\in X$ we have
a $R>0$ such that $B(x,R)=X$ and thus the assumption in the second
part of Theorem \ref{thm3.1} is satisfied by the Walters expanding
conformal map. This implies the formula (\ref{3.7}).\qed

The result is an improvement of Theorem 2.7 in \cite{KotusUrbanski}
which confirms Corollary \ref{cor3.2} under the additional
assumption of that $T$ is strongly regular.  We consider the
existence of the conformal measure and establish the following

\begin{thm}\label{thm3.2}\ \ Let $T:X_0\rightarrow X$ be a Walters expanding conformal map
with expanding constant $C\geq 1$. If $s(T)<\infty$, then
$P(T,s(T))=0$ and furthermore, if
$$C^{(1)}(x,x')=\sup_{T(y)=x}\left|1-{D_dT(y)\over D_dT(y')}\right|\rightarrow 0,\ {\rm as}\
d(x,x')\rightarrow 0,$$ then $T$ has a $D_dT(x)^s$-conformal measure
on $X$.
\end{thm}

{\bf Proof.}\ \  In terms of Lemma \ref{lem3.2} and Theorem
\ref{thm3.1+}, it suffices to prove that $P(T,s(T))=0$. We have
known that $P(T,t)$ is non-increasing in $t$. If for some $t$,
$P(T,t)=0$, then $P(T,s(T))=0$. Therefore we assume that for
arbitrary $t>s(T)$, $P(T,t)<0$ and so $\mathcal{L}_t^n(a)\rightarrow
0$ as $n\rightarrow\infty.$ The following inequality is basic in our
proof:
\begin{eqnarray}\label{3.8}
\mathcal{L}^n_t(a)&=&\sum_{T^n(y)=a}D_dT^n(y)^{-t}\nonumber\\
&=&\sum_{T^{n-1}(w)=a}D_dT^{n-1}(w)^{-t}\sum_{T(y)=w}D_dT(y)^{-t}\nonumber\\
&\geq&D_dT^{n-1}(w)^{-t}\sum_{T(y)=w}D_dT(y)^{-t}\nonumber\\
&=&D_dT^{n-1}(w)^{-t}\mathcal{L}_t(w),\ w\in
T^{-n+1}(a).\end{eqnarray} Take $x_j (1\leq j\leq q)$ such that
$X=\cup_{j=1}^qB(x_j,\delta/2)$ and a $m$ such that for $n>m$,
\begin{equation}\label{3.8+}\mathcal{L}_t^n(a)<1.\end{equation}
Take a $S$ such that for each $j\in\{1,2,...,q\}$ and each $b\in X$,
$T^{-S+1}(b)\cap B(x_j,\delta/2)\not=\emptyset$. From (\ref{3.8+}),
we have $\mathcal{L}_t^S(b)<1$ for some $b\in X$, and hence $b\in
B(x_i,\delta/2)$ for some $i$. Thus, $$\mathcal{L}^S_t(x_i)\leq
M(x_i)^t\mathcal{L}^S_t(b)<M(x_i)^t.$$ From each $T^{-S+1}(x_i)\cap
B(x_j,\delta/2)$ for each $j$, we take a point $w_j^i$ and set
$K(t)=\max\{D_dT^{S-1}(w_j^i)^{t}:1\leq i, j\leq q\}$.

In terms of (\ref{3.8}) with $S$ in the place of $n$, we have
$$\mathcal{L}_t(w_j^i)\leq
D_dT^{S-1}(w_j^i)^{t}\mathcal{L}_t^S(x_i)<K(t)M(x_i)^t.$$ For each
$w\in X$, $w\in B(x_j,\delta/2)$ and so $w\in B(w_j^i,\delta)$ for
some $j$ and then
$$\mathcal{L}_t(w)\leq M(w_j^i)^t\mathcal{L}_t(w_j)<M^{t}(w_j^i)M(x_i)^tK(t).$$
Letting $t\rightarrow s(T)+0$, we have
$$\mathcal{L}_s(w)\leq \max\{M^{s}(w_j^i):1\leq i,j\leq q\}K(s).$$
We have proved that $\varphi_s=-s\log D_dT(x)$ with $s=s(T)$ is
summable on $X$ so that $P(T,s)\leq 0$. This immediately implies
that $P(T,s(T))=0$. \qed

Combining Corollary \ref{cor3.2} and Theorem \ref{thm3.2} deduces
that the Bowen formula holds, i.e., $P(T,s(T))=0$ and $s(T)={\rm
dim}_H(X_\infty)$ for a Walters expanding conformal map
$T:X_0\rightarrow X$ with the expanding constant $C>1$ and
$s(T)<\infty$.

The following result confirms the existence of invariant measure
which is equivalent to the conformal measure.

\begin{thm}\label{thm3.4}\ \ Let $T:X_0\rightarrow X$ be a Walters expanding conformal map
with expanding constant $C>1$ or $C=1$ and $\varrho=1$, and
$s(T)<\infty$. If
$$C(x,x')=\sup_{n\geq 1}\sup_{T^n(y)=x}\left|1-{D_dT^n(y)\over D_dT^n(y')}\right|\rightarrow 0,\ {\rm as}\
d(x,x')\rightarrow 0,$$ then there exist a $s$-conformal measure
$\mu_s$ and an invariant Gibbs measure $m_s$ which are equivalent
and furthermore, the statements listed in Theorem \ref{thm2.6}
hold.\end{thm}

Theorem \ref{thm3.4} is attained by applying Lemma \ref{lem3.2},
Theorem \ref{thm3.2} and Theorem \ref{thm1.1}. The existence of
$\mu_s$ and $m_s$ was stated by Kotus and Urbanski in
\cite{KotusUrbanski} with $C>1$ (and N=1) for $X\subset \mathbb{C}$
and $T$ being regular, namely, $P(T,s)=0$, as in this case,
$\varphi_s=-s\log D_dT(x)$ is dynamically H\"older continuous in
view of the Koebe's distortion theorem.

\

\end{document}